\documentclass[reqno,12pt]{amsart}
\newcommand{\ben}{\begin{enumerate}}
\newcommand{\een}{\end{enumerate}}
\newcommand{\bqu}{\begin{quote}}
\newcommand{\equ}{\end{quote}}
\newcommand{\beq}{\begin{equation}}
\newcommand{\eeq}{\end{equation}}
\newcommand{\bec}{\begin{center}}
\newcommand{\ece}{\end{center}}

\usepackage{url}
\usepackage{latexsym,epsfig,amssymb,amsmath,amsthm,color,url}
\usepackage[shortlabels]{enumitem}
\usepackage{hyperref}

\allowdisplaybreaks 
 \numberwithin{equation}{section}
\newtheorem{theorem}{Theorem}[section]

\newtheorem{lemma}[theorem]{Lemma}

\DeclareMathOperator{\col}{col}
\DeclareMathOperator{\EHR}{EHR}
\DeclareMathOperator{\DEHR}{DEHR}
\DeclareMathOperator{\Types}{TYPES}
\DeclareMathOperator{\Type}{Type}

\theoremstyle{definition}
\newtheorem{defn}[theorem]{Definition}

\theoremstyle{remark}
\newtheorem{rem}[theorem]{Remark}

\title{Second order logic on random rooted trees}
\date{}
\author{Alexander Holroyd}
\address{Alexander E.\ Holroyd, \ Microsoft Research, \ 1 Microsoft Way, Redmond, WA 98052, United States.}
\email{holroyd@microsoft.com.}
\author{Avi Levy}
\address{Avi Levy, \ University of Washington, \ Box 354350, Seattle, WA 98195-4350, United States.}
\email{avius@uw.edu.}
\author{Moumanti Podder}
\address{Moumanti Podder, \ Courant Institute of Mathematical Sciences, \ New York University, \ 251 Mercer Street, New York, NY 10012, United States.}
\email{mp3460@nyu.edu.}
\author{Joel Spencer}
\address{Joel Spencer, \ Courant Institute of Mathematical Sciences, \ New York University, \ 251 Mercer Street, New York, NY 10012, United States.}
\email{spencer@cims.nyu.edu.}

\begin{document}
\bibliographystyle{plain}
\nocite{*}

\begin{abstract}
We address questions of logic and expressibility in the context of random rooted trees. Infiniteness of a rooted tree is not expressible as a first order sentence, but is expressible as an existential monadic second order sentence (EMSO). On the other hand, finiteness is not expressible as an EMSO. For a broad class of random tree models, including Galton-Watson trees with offspring distributions that have full support, we prove the stronger statement that finiteness does not agree up to a null set with any EMSO. We construct a finite tree and a non-null set of infinite trees that cannot be distinguished from each other by any EMSO of given parameters. This is proved via set-pebble Ehrenfeucht games (where an initial colouring round is followed by a given number of pebble rounds).
\end{abstract}



\maketitle

\section{Introduction}\label{intro}
\sloppy The problem of expressibility in a given language parametrized by mathematical logic has been one of the classically studied problems. In our paper, the setting is the space of all rooted trees, and the language is the set of all existential monadic second order sentences on rooted trees. The property we specifically focus on is the finiteness of the rooted tree. 
\par Before we go into the details of the questions and how we seek to answer them, we point out here that such questions have important implications in descriptive complexity theory. Descriptive complexity measures the syntactic complexity of formulae that express a certain property, instead of its computation complexity. See \cite{immerman} and \cite{thomas} for more general discussions on this theory. A fundamental result in this area is the well-known \emph{Fagin's theorem}, which states that a property is in the class NP (non-deterministic polynomial time computability) if and only if it is describable as an existential second-order logical sentence (see \cite{fagin}). If one can prove that the class of all existential second-order sentences is not closed under negation, one shall establish that NP $\neq$ co-NP and therefore P $\neq$ NP. 
\par It is possible to express the property of infiniteness of the rooted tree, in a simple way, as an existential monadic second order sentence (EMSO) (see \cite{courcelle2}, \cite{courcelle1}, \cite{gottlob} and \cite{vaananen} for more on EMSO). This naturally raises the question as to whether finiteness of the rooted tree, the negation of infiniteness, can also be expressed as an EMSO. The objective of this paper is to tie in probability with this question. We ask if it is possible, under a measure $\mu$ that satisfies certain naturally occurring conditions, that finiteness is expressible as an EMSO on all but a subset of trees of measure $0$. We answer this question in the negative, and the result is stated in Theorem~\ref{main}. As a straightforward consequence of Theorem~\ref{main}, one can conclude that finiteness is not expressible tautologically as an EMSO. 

\begin{theorem}\label{main}
Let $\mu$ be a probability measure on $\mathcal{T}$ that assigns positive probability to every set of the form $\{T: T|_{n} = T'\}$, where $n \in \mathbb{N}$, and $T|_{n}$ denotes the truncation to depth $n$, and $T'$ is a finite tree. Then, the property of finiteness of rooted trees is not \emph{almost expressible} as an existential monadic second order sentence under the measure $\mu$. That is, there do not exist a $\mu$-null set of infinite trees $\mathcal{T}_{0}$ and an EMSO $A$, such that every finite tree satisfies $A$ and every infinite tree in $\mathcal{T}_{0}^{c}$ satisfies the negation of $A$.
\end{theorem}
Such a measure assigns positive weight to every finite tree. In particular, we can consider the measure induced by the well-known Galton-Watson branching process with an offspring distribution $\chi$ that is supported on all of $\mathbb{N}_{0}$, the set of non-negative integers, with expectation greater than $1$ (i.e.\ the supercritical regime). An example is the Poisson distribution with expectation greater than $1$. In the following subsection, we set down the notations we use throughout the paper. 

\subsection{Some notation}\label{notation}
We denote by $T_{\mu}$ the random rooted tree which follows the measure $\mu$. For any tree $T \in \mathcal{T}$, we let $V(T)$ denote its set of nodes. For any $v \in V(T)$, we let $d(v)$ denote the depth of $v$ in $T$, where the root, usually denoted $\phi$, has depth $d(\phi) = 0$. For $v \in V(T)$, let $T(v)$ denote the subtree of $T$ that is rooted at $v$. When $v$ is a child of the root, we call $T(v)$ a \emph{principal branch} of $T$. Let $\pi(v)$ denote the parent of $v$, for any $v \in V(T) \setminus \{\phi\}$. For a positive integer $n$ and $T \in \mathcal{T}$, let $T|_{n}$ denote the truncation of $T$, consisting of all nodes of depth at most $n$. For a positive integer $k$, we set $[k] = \{0, 1, \ldots, k\}$. 

\subsection{EMSO on trees}
Existential monadic second order (EMSO) sentences on $\mathcal{T}$ are of the form
$$\exists \ S_{1}, \ldots, \exists \ S_{n} [P],$$
where $S_{1}, \ldots, S_{n}$ are subsets of nodes of the tree, and $P$ is a first-order sentence that involves the root as a constant symbol and the relations $=$ (equality of nodes), $\pi$ (parent-child relationship) and $\in$ (inclusion in one of the subsets $S_{1}, \ldots, S_{n}$). Often, it is more easily visualizable if we identify the subsets $S_{1}, \ldots, S_{n}$ with colours, i.e.\ we partition the set of all rooted trees into $n$ colour classes. A classical example would be the infiniteness of the tree, which is expressible as follows:
\begin{equation}\label{infiniteness EMSO}
\exists \ S \Big[\big[\phi \in S\big] \wedge \big[\forall \ u \in S \left[\exists \ v \in S [\pi(v) = u]\right]\big]\Big].
\end{equation}
In words, this asserts that there exists a set $S$ of nodes containing the root, such that every element $u$ of $S$ has a child $v$ in $S$. As mentioned earlier, this paper is concerned with showing that the complementary event i.e.\ that the tree dies out, is not expressible as an EMSO almost surely under any probability measure $\mu$ on $\mathcal{T}$ that satisfies the hypotheses of Theorem~\ref{main}. Our proof technique relies on a suitable version of the well-known Ehrenfeucht games, the set-pebble Ehrenfeucht games (see Definition~\ref{set-pebble EHR}). 

\section{Rooted colourings, set-pebble Ehrenfeucht and types games}\label{prelim}
In this paper, we fix an arbitrary positive integer $r$ and consider a set $\Sigma = \{\col_{0}, \ldots, \col_{r}\}$ of $r+1$ colours. Later on, we shall consider a set $\overline{\Sigma}$ of ``augmented" colours that is derived from $\Sigma$. 
\begin{defn}[Rooted colouring]\label{rooted colouring}
Given the set $\Sigma$ of colours, and a tree $T \in \mathcal{T}$, we call a colouring $\sigma: V(T) \rightarrow \Sigma$ a \emph{$(\Sigma, \col_{0})$-rooted colouring} of $T$ if $\sigma(v) = \col_{0} \Leftrightarrow v = \phi$, for all $v \in V(T)$.
\end{defn}
We insist upon assigning a unique colour to the root because it is a constant symbol in our language. Given a $T \in \mathcal{T}$ and a $(\Sigma, \col_{0})$-rooted colouring $\sigma$ of $T$, we shall talk of the pair $(T, \sigma)$ in the subsequent exposition, and call it a \emph{$(\Sigma, \col_{0})$-coloured tree}.

\par As previously mentioned, the set-pebble Ehrenfeucht game, defined below, will be our main tool in proving that indeed, there is a set of trees of positive measure on which finiteness is not an EMSO. A generalized version of this game is described in Definition~7.17 of \cite{libkin}, where it is referred to as the \emph{Fagin game}. 

\par In all games considered in this paper, when we say that a particular player wins, we mean with optimal play by both the players. In other words, the player who wins has a strategy that guarantees a win regardless of the other player's moves.

\begin{defn}[The set-pebble Ehrenfeucht game]\label{set-pebble EHR}
This game is played between two players, Spoiler and Duplicator. They are given two trees $T_{1}, T_{2}$ with roots $\phi_{1}, \phi_{2}$, and a positive integer $k$. The game consists of $k+1$ rounds, and each round consists of a move by Spoiler and a subsequent move by Duplicator. These rounds can be divided into the following:
\begin{enumerate}
\item \label{set round} \textbf{Set round:} In an initial round, called the \emph{set round}, Spoiler assigns a $(\Sigma, \col_{0})$-rooted colouring $\sigma_{1}$ to $T_{1}$. In reply, Duplicator assigns a $(\Sigma, \col_{0})$-rooted colouring $\sigma_{2}$ to $T_{2}$.
\item \label{pebble rounds} \textbf{Pebble rounds:} The subsequent rounds, numbered $1$ through $k$, are called \emph{pebble rounds}. In each of these rounds, Spoiler chooses \emph{either} of the two trees $T_{1}$ and $T_{2}$ and selects a node from that tree. In reply, Duplicator has to select a node from the other tree. 
\end{enumerate}
Suppose $x_{i}$ is the node selected from $T_{1}$, and $y_{i}$ the node selected from $T_{2}$, in round $i$, for $1 \leq i \leq k$. By convention, we set
\begin{equation}\label{convention}
x_{0} = \phi_{1} \text{ and } y_{0} = \phi_{2}.
\end{equation}
Then Duplicator wins this game, denoted $\EHR\left[T_{1}, T_{2}, \Sigma, k\right]$, if \emph{all} of the following conditions are satisfied: for all $i, j \in [k]$ (with convention \eqref{convention}),
\begin{enumerate}[label={(EHR \arabic*)},leftmargin=*]
\item \label{EHR 1} $\pi(x_{j}) = x_{i} \Leftrightarrow \pi(y_{j}) = y_{i}$;
\item \label{EHR 2} $\sigma_{1}(x_{i}) = \sigma_{2}(y_{i})$;
\item \label{EHR 3} $x_{i} = x_{j} \Leftrightarrow y_{i} = y_{j}$.
\end{enumerate}
\end{defn}

The following theorem leads to the conclusion of the main result, Theorem~\ref{main}. 
\begin{theorem}\label{main 1}
Fix positive integers $k$ and $r$, and the set of $r$ colours $\Sigma$. Then there exist a finite tree $T_{1}$, and a family $\mathcal{T}_{2}$ of infinite trees, such that $\mu\left[\mathcal{T}_{2}\right] > 0$ and Duplicator wins $\EHR[T_{1}, T_{2}, \Sigma, k]$ for each $T_{2} \in \mathcal{T}_{2}$.
\end{theorem}
We prove Theorem~\ref{main} as a consequence of Theorem~\ref{main 1}, as follows.
\begin{proof}[Proof of Theorem \ref{main}]
For this, we refer to Proposition~7.18 of \cite{libkin}. This proposition gives a necessary and sufficient condition for the expressibility of a given property $P$, defined on a general set of structures, as an EMSO sentence over such structures. It states that $P$ is not expressible as an EMSO if and only if, given any two positive integers $r, k$, there exist a structure $A$ that satisfies $P$, and a structure $B$ that does not, such that Duplicator wins the set-pebble Ehrenfeucht game, with parameters $r$ and $k$, on $A$ and $B$. The proposition does not involve any probability (i.e.\ only tautological expressibility considered). However, we show that Duplicator wins the set-pebble Ehrenfeucht game on trees $T_{1}, T_{2}$ that are as described in Theorem~\ref{main 1}. This shows that tautologically, finiteness is not expressible as an EMSO. Since we additionally show that this holds for every $T_{2} \in \mathcal{T}_{2}$, and $\mu\left[\mathcal{T}_{2}\right] > 0$, we can conclude that finiteness is not almost surely expressible as an EMSO under the measure $\mu$.
\end{proof}
For more general discussion on the connection between Ehrenfeucht games and EMSO properties, see also \cite{immerman} and \cite{marker}. 

\par In order to prove Theorem~\ref{main 1}, we introduce the \emph{types game} in Definition~\ref{type game}, which is \emph{harder} than the set-pebble Ehrenfeucht game in the following sense: if Duplicator is able to win the types game with certain parameters on given trees $T_{1}, T_{2}$, she also wins the set-pebble Ehremfeucht game on $T_{1}, T_{2}$ with related parameters (these are made precise in the sequel). Here and elsewhere, out terminology is biased in favour of Duplicator, since our ultimate goal is to prove that she can win. The proof of Theorem~\ref{main 1} happens in two steps. In Section~\ref{type to EHR}, we show that, given two trees $T_{1}, T_{2}$ which satisfy certain assumptions on the neighbourhoods of their roots, if Duplicator wins the types game with suitable parameters on $T_{1}, T_{2}$, then she also wins the set-pebble Ehrenfeucht game with related parameters. In Section~\ref{type win positive probab}, we show the construction of a finite tree $T_{1}$, and a family $\mathcal{T}_{2}$ of infinite trees, with $\mu[\mathcal{T}_{2}] > 0$, such that Duplicator is able to win the types game on $T_{1}$ and $T_{2}$, for every $T_{2} \in \mathcal{T}_{2}$. Combining the results of these two sections, we derive the final conclusion of Theorem~\ref{main 1}.

\par The definition of the types game requires us to first define the \emph{type} of a node. A type is defined in terms of $\Sigma$, a \emph{cut-off} $k \in \mathbb{N}$, and depth $m \in \mathbb{N}$. These same parameters are also required when we define the types game. 

\begin{defn}[Type]\label{type def}
Given a $(\Sigma, \col_{0})$-coloured tree $(T, \sigma)$, we define the $(\Sigma, m, k)$-type $\Type_{\Sigma,m,k,(T,\sigma)}(v) = \Type_m(v)$ of any $v \in V(T)$ recursively on $m$, as follows. The $(\Sigma, 0, k)$-type  $\Type_{0}(v)$ of $v$ with respect to $(T, \sigma)$ is simply the colour $\sigma(v)$ of $v$.
\par Suppose we have defined $\Type_{m-1}(v)$ for all $v \in V(T)$. Let $\Gamma_{m-1} = \Gamma_{\Sigma, m-1, k}$ denote the set of all possible $(\Sigma, m-1, k)$ types. For any $v \in V(T)$, if $n_{\gamma}$ denotes the number of children $u$ of $v$ with $\Type_{m-1}(u) = \gamma$, we let $\vec{n} = (n_{\gamma} \wedge k: \gamma \in \Gamma_{m-1})$. The $(\Sigma, m, k)$-type of $v$ with respect to $(T, \sigma)$ is then given by
\begin{equation}
\Type_{m}(v) = (\sigma(v), \vec{n}).
\end{equation}
\end{defn}
Note the use of the minimum with the cutoff $k$ in the definition of $\vec{n}$ above -- we count the number of children of each type only up to the cutoff $k$.

When the set of colours $\Sigma$, the tree $T$ and its colouring $\sigma$, and the parameters $m$ and $k$ are clear from the context, we shall call the $(\Sigma, m, k)$-type of a node $v$ in $V(T)$ its ``depth $m$ type" or simply ``type". If there are types of varying depths being considered, we shall denote the set of all possible depth $m$ types by the shortened notation $\Gamma_{m}$ as opposed to $\Gamma_{\Sigma, m, k}$; if the depth is also clear from the context, we shall simply refer to it as $\Gamma$. When there is no chance of confusion, we shall also drop the subscripts from $\Type_{\Sigma, m, k, (T, \sigma)}(v)$ and denote the type of $v$ as $\Type(v)$.  
 
We now describe the types game. This is a single round game played between Spoiler and Duplicator. 
\begin{defn}[Types game]\label{type game}
Fix $\Sigma, m, k$. The players are given two trees $T_{1}$ and $T_{2}$. Spoiler assigns a $(\Sigma, \col_{0})$-rooted colouring $\sigma_{1}$ to $T_{1}$. In reply, Duplicator assigns a $(\Sigma, \col_{0})$-rooted colouring $\sigma_{2}$ to $T_{2}$. For each type $\gamma \in \Gamma_{\Sigma, m, k}$, let $n_{\gamma}^{(1)}$ and $n_{\gamma}^{(2)}$ denote the numbers of nodes of type $\gamma$ in $(T_{1}, \sigma_{1})$ and $(T_{2}, \sigma_{2})$ respectively. Duplicator wins this game denoted by $\Types[T_{1}, T_{2}, \Sigma, k, m]$, if 
\begin{equation}\label{win type}
n_{\gamma}^{(1)} \wedge k = n_{\gamma}^{(2)} \wedge k \quad \text{for all } \gamma \in \Gamma.
\end{equation}
\end{defn}

\section{Enhanced colouring and distance preserving Ehrenfeucht games}\label{type to EHR prelim}
\par One of the two main steps in the proof of Theorem \ref{main 1} will be to show that if Duplicator wins the types game with certain parameters on two given trees $T_{1}, T_{2}$, she wins the set-pebble Ehrenfeucht game with related parameters on those trees. This is stated and proved in Theorem~\ref{main 2} in Section~\ref{type to EHR}. In this section, we have several different tools to introduce and discuss, so as to use them in the proof of Theorem~\ref{main 2}. These are done in separate subsections. In Subsection~\ref{enhanced colouring subsec}, we introduce the notion of \emph{enhanced colouring}, which lets Duplicator attach an extra label to each node of a $(\Sigma, \col_{0})$-coloured tree. This additional information helps her choose a judicious colouring in the set round. In Subsection~\ref{DEHR subsec}, we introduce yet another version of the Ehrenfeucht games, called the distance preserving Ehrenfeucht game (DEHR). The DEHR is different from the set-pebble Ehrenfeucht game in that it does not involve a set round, as the trees provided are already coloured. Moreover, Duplicator needs to maintain the graph distances between corresponding pairs of chosen nodes on the two trees (i.e.\ the language now also involves the relation $\rho$, which is the shortest edge distance between two nodes in a graph). 

\subsection{Enhanced colouring:}\label{enhanced colouring subsec}
Given a colouring of the nodes of a tree $T$ using the colours in $\Sigma$, we discuss here a method of ``refining" this colouring by attaching a second marker to each node. This \emph{enhanced colouring} will be of use to Duplicator in her winning strategy for Theorem~\ref{main 2}. 
\begin{defn}\label{en col def}
Fix positive integers $D, D_{0}$, with $D$ even, and $D_{0}+1$ many new colours $\col'_{0}, \ldots, \col'_{D_{0}}$ which are different from all the colours of $\Sigma$. Suppose we are given a $(\Sigma, \col_{0})$-coloured tree $(T, \sigma)$. Let 
\begin{equation}\label{en col set}
\overline{\Sigma} = \Big\{\big(\col_{i}, \col'_{j}\big): i \in [r], j \in [D_{0}]\Big\} \cup \Big\{\big(\col_{i}, j\big): i \in [r], -D/2 < j \leq D/2\Big\}.
\end{equation}
Further define the function $F: \mathbb{N} \cup \{0\} \rightarrow \{\col'_{0}, \ldots, \col'_{D_{0}}\} \cup \{-D/2+1, \ldots, D/2\}$ as
\begin{equation} \label{depth fn}
F(i) = 
  \begin{cases} 
   \col'_{i} & \text{if } 0 \leq i \leq D_{0}, \\
   i \bmod D & \text{if } i > D_{0},
  \end{cases}
\end{equation}
where, here and subsequently, $i \bmod D$ denotes unique element of $\{-D/2+1, \ldots, D/2\}$ that is congruent to $i$ modulo $D$. Then we define the \emph{enhanced colouring corresponding to $\sigma$, with respect to the parameters $\col'_{0}, \ldots, \col'_{D_{0}}, D$}, to be the map $\overline{\sigma}: V(T) \rightarrow \overline{\Sigma}$ given by:
\begin{equation}\label{en col}
\overline{\sigma}(v) = \Big(\sigma(v), F\big(d(v)\big)\Big), \quad \text{for all } v \in V(T),
\end{equation}
where recall that $d(v)$ is the depth of $v$ in $T$. 
\end{defn}
We call the elements of $\overline{\Sigma}$ \emph{augmented colours}. Henceforth, we do not mention the parameters $D$, $\col'_{0}$, $\ldots$, $\col'_{D_{0}}$, except to state the specific values of $D_{0}$ and $D$ in the statement of the results where they are required. We define a colouring $\sigma': V(T) \rightarrow \overline{\Sigma}$ to be \emph{legal} if there exists a $(\Sigma, \col_{0})$-rooted colouring $\sigma$ such that $\sigma' = \overline{\sigma}$. Clearly, a legal colouring is also a $\big(\overline{\Sigma}, (\col_{0}, \col'_{0})\big)$-rooted colouring.

\begin{lemma}\label{legal}
Suppose $T_{1}, T_{2}$ are two given trees with roots $\phi_{1}, \phi_{2}$, such that Duplicator wins $\Types[T_{1}, T_{2}, \overline{\Sigma}, k, m]$ for some positive integers $k, m$. For any legal colouring $\sigma_{1}$ that Spoiler assigns to $T_{1}$, the winning reply of Duplicator must be a legal colouring $\sigma_{2}$ of $T_{2}$.
\end{lemma}
\begin{proof}
As Spoiler assigns legal colouring $\sigma_{1}$, hence $\sigma_{1}\left(\phi_{1}\right) = (\col_{0}, \col'_{0})$. Recall from Definition~\ref{type game} that Duplicator needs to assign a $\big(\overline{\Sigma}, (\col_{0}, \col'_{0})\big)$-rooted colouring to $T_{2}$, hence we must have $\sigma_{2}\left(\phi_{2}\right) = (\col_{0}, \col'_{0})$. Suppose now the claim of Lemma~\ref{legal} were false. We can then find a smallest positive integer $s$, and a node $v \in V(T_{2})$ with $d(v) = s$, such that $\sigma_{2}(v) = (\col_{i}, j)$ with $j \neq F(s)$, for some $1 \leq i \leq r$. The parent $u$ of $v$ will satisfy $\sigma_{2}(u) = (\col_{i'}, F(s-1))$, for some $0 \leq i' \leq r$ ($i' = 0$ iff $u = \phi_{2}$). 
\par Since $\sigma_{1}$ is legal, for any $w \in V(T_{1})$ with $\sigma_{1}(w) = \big(\col_{i'}, F(s-1)\big)$, we must have $F(d(w)) = F(s-1)$; and if $w$ has any child $w'$, then the second coordinate of $\sigma_{1}(w')$ will be $F(s)$. Hence the depth $1$ type of $u$ in $(T_{2}, \sigma_{2})$ does not agree with that of any node in $(T_{1}, \sigma_{1})$, hence nor does its depth $m$ type for any positive integer $m$. But this means, by \eqref{win type}, that $\sigma_{2}$ is not a winning move for Duplicator. Thus we have arrived at a contradiction.
\end{proof}

\subsection{The distance preserving Ehrenfeucht game (DEHR), and its relation to the types game}\label{DEHR subsec}
The DEHR game is a tool that is used by Duplicator to find a winning strategy in Theorem~\ref{main 1}. It is used as a local tool, in order to find her responses in small subtrees inside the given trees. This is made precise in the proof of Theorem~\ref{main 1} in Section~\ref{type to EHR}. To define the DEHR, we first fix a positive integer $k$ and the set of colours $\Sigma$. 

\begin{defn}[DEHR]\label{DEHR}
We are given $(\Sigma, \col_{0})$-coloured trees $(T_{1}, \sigma_{1})$ and $(T_{2}, \sigma_{2})$, with roots $\phi_{1}, \phi_{2}$, and pairs $(x_{i}, y_{i}) \in V(T_{1}) \times V(T_{2})$, $1 \leq i \leq \ell$, for some $0 \leq \ell \leq k$. We may think of these pairs as earlier moves already played. There are $k - \ell$ rounds in the game. In particular, when $\ell = 0$, we are given no such pair of nodes, and when $\ell = k$, there are no rounds to be played in the game. When $\ell = 0$, we denote the game by $\DEHR\big[(T_{1}, \sigma_{1}), (T_{2}, \sigma_{2}), \Sigma, k\big]$, otherwise we denote it by $\DEHR\big[(T_{1}, \sigma_{1}), (T_{2}, \sigma_{2}), \Sigma, k, \left\{(x_{i}, y_{i}): 1 \leq i \leq \ell\right\}\big]$.
\par In each of the $k-\ell$ rounds, Spoiler picks a node from \emph{either} of the two trees, and in reply, Duplicator picks a node from the other tree. Let $x_{j+\ell}$ be the node selected from $T_{1}$ and $y_{j+\ell}$ that from $T_{2}$ in round $j$, for $1 \leq j \leq k-\ell$. We again follow the convention given in \eqref{convention}, i.e.\ set $x_{0}, y_{0}$ to be the respective roots. Duplicator wins this game if \emph{all} of the following conditions are satisfied: for all $i, j \in [k]$ (with convention \eqref{convention}),
\begin{enumerate}[label={(DEHR \arabic*)},leftmargin=*]
\item \label{DEHR win 1} $\rho(x_{i}, x_{j}) = \rho(y_{i}, y_{j})$ where $\rho$ is the usual graph distance;
\item \label{DEHR win 2} $\pi(x_{j}) = x_{i} \Leftrightarrow \pi(y_{j}) = y_{i}$;
\item \label{DEHR win 3} $\sigma_{1}(x_{i}) = \sigma_{2}(y_{i})$;
\item \label{DEHR win 4} $x_{i} = x_{j} \Leftrightarrow y_{i} = y_{j}$.
\end{enumerate}
\end{defn}
When $\ell > 0$, the given pairs $(x_{i}, y_{i}), 1 \leq i \leq \ell$, are often referred to as \emph{designated pairs} in the game. Observe that it only makes sense to consider designated pairs that themselves satisfy Conditions \ref{DEHR win 1} through \ref{DEHR win 4}, else Duplicator has no chance of winning the game. 

\par For $1 \leq \ell \leq k$, a collection of designated pairs $\big\{(x_{i}, y_{i}): 1 \leq i \leq \ell\big\}$ is called a \emph{winnable} configuration for $\big\{(T_{1}, \sigma_{1}), (T_{2}, \sigma_{2}), \Sigma, k\big\}$ if Duplicator wins $\DEHR\big[(T_{1}, \sigma_{1}), (T_{2}, \sigma_{2}), \Sigma, k, \left\{(x_{i}, y_{i}): 1 \leq i \leq \ell\right\}\big]$ (with optimal play by both players, as usual). When $\ell = k$, we often call such a configuration simply \emph{winning}.

\begin{lemma}\label{obs imp later}
Suppose we are given $(\Sigma, \col_{0})$-coloured trees $(T_{1}, \sigma_{1})$ and $(T_{2}, \sigma_{2})$, positive integer $k$, and a winnable configuration $\left\{(x_{i}, y_{i}): 1 \leq i \leq \ell\right\}$ for $\big\{(T_{1}, \sigma_{1}), (T_{2}, \sigma_{2}), \Sigma, k\big\}$ for some $1 \leq \ell \leq k$. Then all $i, j \in [\ell]$ satisfy Conditions \ref{DEHR win 1} through \ref{DEHR win 4}.
\end{lemma}

\begin{proof}
When $\left\{(x_{i}, y_{i}): 1 \leq i \leq \ell\right\}$ is winnable for $\big\{(T_{1}, \sigma_{1}), (T_{2}, \sigma_{2}), \Sigma, k\big\}$, Duplicator wins $\DEHR\big[(T_{1}, \sigma_{1}), (T_{2}, \sigma_{2}), \Sigma, k\big]$ where the first $\ell$ rounds were $(x_{i}, y_{i}), 1 \leq i \leq \ell$, and the $0$-th round was $(x_{0}, y_{0}) = (\phi_{1}, \phi_{2})$. Hence they have to satisfy Duplicator's winning conditions for DEHR, i.e.\ Conditions \ref{DEHR win 1} through \ref{DEHR win 4}.
\end{proof}

\begin{defn}\label{cor node}
Suppose the $(\Sigma, \col_{0})$-coloured trees $(T_{1}, \sigma_{1})$ and $(T_{2}, \sigma_{2})$ are such that Duplicator wins $\DEHR\big[(T_{1}, \sigma_{1}), (T_{2}, \sigma_{2}), \Sigma, k\big]$. For any $u \in V(T_{1})$, we define $v \in V(T_{2})$ to be a \emph{corresponding node} to $u$ if, for $x_{1} = u$ and $y_{1} = v$, the configuration $\{(x_{1}, y_{1})\}$ is winnable for $\big\{(T_{1}, \sigma_{1}), (T_{2}, \sigma_{2}), \Sigma, k\big\}$. Symmetrically, we can define a corresponding node in $V(T_{1})$ to any $v \in V(T_{2})$. 
\par Suppose we are given $(T_{1}, \sigma_{1})$, $(T_{2}, \sigma_{2})$ and a winnable configuration $\big\{(x_{i}, y_{i}): 1 \leq i \leq \ell\big\}$ for $\big\{(T_{1}, \sigma_{1}), (T_{2}, \sigma_{2}), \Sigma, k\big\}$, for some $1 \leq \ell \leq k-1$. For any $u \in V(T_{1})$, we define $v \in V(T_{2})$ to be a \emph{corresponding node} to $u$ if the configuration $\big\{(x_{i}, y_{i}), 1 \leq i \leq \ell+1\big\}$ with $x_{\ell+1} = u$ and $y_{\ell+1} = v$, is winnable for $\big\{(T_{1}, \sigma_{1}), (T_{2}, \sigma_{2}), \Sigma, k\big\}$. Symmetrically, we can define a corresponding node in $V(T_{1})$ to any $v \in V(T_{2})$. 
\end{defn}

\par Note that the choice of a corresponding node need not be unique. But if $\{(x_{i}, y_{i}): 1 \leq i \leq \ell\}$ is winnable for $\DEHR\big[(T_{1}, \sigma_{1}), (T_{2}, \sigma_{2}), \Sigma, k\big]$ \Big(or in the case of $\ell = 0$, Duplicator wins $\DEHR\big[(T_{1}, \sigma_{1}), (T_{2}, \sigma_{2}), \Sigma, k\big]$\Big), every node in $V(T_{1})$ \big(correspondingly $V(T_{2})$\big) will have at least one corresponding node in the other tree. 

\par The following lemma establishes a crucial connection between the types game and the DEHR. This connection will be utilized in devising a winning strategy for Duplicator in the proof of Theorem~\ref{main 2}.
\begin{lemma}\label{same type win DEHR}
Let $(T_{1}, \sigma_{1})$ and $(T_{2}, \sigma_{2})$ be two given $(\Sigma, \col_{0})$-coloured trees. If nodes $u_{1} \in V(T_{1})$ and $u_{2} \in V(T_{2})$ have the same $(\Sigma, m k)$-type, then Duplicator wins $$\DEHR\big[\left(T_{1}(u_{1})|_{m}, \sigma_{1}\right), \left(T_{2}(u_{2})|_{m}, \sigma_{2}\right), \Sigma, k\big].$$
\end{lemma}
Here, $\left(T_{1}|_{m}, \sigma_{1}\right)$ \big(respectively $\left(T_{2}|_{m}, \sigma_{2}\right)$\big) denotes the truncated subtree $T_{1}(u_{1})|_{m}$ \big(respectively $T_{2}(u_{2})|_{m}$\big) along with the colouring $\sigma_{1}$ \big(respectively $\sigma_{2}$\big) restricted to this truncated subtree.
\par Note that, under the stronger assumption that $u_{1}$ and $u_{2}$ have the same $(\Sigma, m, \infty)$-types (so that counts are not truncated at any finite cutoff), the conclusion of the lemma for the DEHR with finitely many rounds would be obvious, since the two trees $\left(T_{1}(u_{1})|_{m}, \sigma_{1}\right)$ and $\left(T_{2}(u_{2})|_{m}, \sigma_{2}\right)$ would be isomorphic (where the isomorphism maintains both adjacency and colour of nodes).

\par The proof of Lemma~\ref{same type win DEHR} is rather technical, but the idea is straightforward. We use induction on $m$. Assuming the lemma holds for $m$, let $u_{1}, u_{2}$ have the same depth-$(m+1)$ types. Then, for any depth-$m$ type $\gamma$, the numbers of principal branches (recall definition from Subsection \ref{notation}) of type $\gamma$ in $T_{1}(u_{1})|_{m+1}$ and $T_{2}(u_{2})|_{m+1}$ are equal when counted up to the cutoff $k$. If the numbers are exactly equal in both trees, Duplicator simply plays in the $i$-th type-$\gamma$ principal branch in $T_{1}(u_{1})|_{m+1}$ when Spoiler plays in the $i$-th type-$\gamma$ principal branch in $T_{2}(u_{2})|_{m+1}$. She plays according to her winning strategy on two depth-$m$ trees, both of type $\gamma$. When there are at least $k$ type-$\gamma$ principal branch in each tree, Duplicator chooses from a ``new" type-$\gamma$ principal branch (i.e.\ one from which no nodes have been chosen so far) in $T_{1}(u_{1})|_{m+1}$ every time Spoiler does the same in $T_{2}(u_{2})|_{m+1}$, and vice versa. This works because the number of rounds in the game is $k$. Her selections again follow her winning strategy on two depth-$m$ trees of type $\gamma$.

\begin{proof}
We prove Lemma~\ref{same type win DEHR} via induction on $m$. Th case $m = 0$ is immediate. Suppose it holds for some $m \geq 0$. Let $u_{1} \in V(T_{1})$ and $u_{2} \in V(T_{2})$ have the same depth-$(m+1)$ types. For this proof, we abbreviate $\Gamma_{\Sigma, i, k}$ by $\Gamma_{i}$ for all $i$. For all $\gamma \in \Gamma_{m}$, let $n^{(1)}_{\gamma}$ and $n^{(2)}_{\gamma}$ be the numbers of children, with depth-$m$ type $\gamma$, of $u_{1}$ and $u_{2}$ respectively. From the definition of types, we have
\begin{equation}\label{same type count}
n^{(1)}_{\gamma} \wedge k = n^{(2)}_{\gamma} \wedge k \text{ for all } \gamma \in \Gamma_{m}.
\end{equation}
Moreover, $\sigma_{1}(u_{1}) = \sigma_{2}(u_{2})$. Recall from Subsection~\ref{notation} that a principal branch is a subtree originating from one of the children of the root. Let $S_{\gamma, i}$, for $1 \leq i \leq n^{(1)}_{\gamma}$, be the principal branches of $T_{1}\left(u_{1}\right)|_{m+1}$ with depth-$m$ type $\gamma$. Let $T_{\gamma, j}$, for $1 \leq j \leq n^{(2)}_{\gamma}$, be the principal branches of $T_{2}\left(u_{2}\right)|_{m+1}$ with depth-$m$ type $\gamma$. By the induction hypothesis, for all $\gamma \in \Gamma_{m}$,
\begin{equation}\label{ind hyp same type}
\text{Duplicator wins } \DEHR\Big[\left(S_{\gamma, i}, \sigma_{1}\right), \left(T_{\gamma, j}, \sigma_{2}\right), \Sigma, k\Big], 
\end{equation}
for all $1 \leq i \leq n^{(1)}_{\gamma}$ and $1 \leq j \leq n^{(2)}_{\gamma}$. As in convention \eqref{convention}, set the $0$-th pair to be the roots, i.e.\ $x_{0} = u_{1}$ and $y_{0} = u_{2}$. Suppose $s$ rounds of the game have been played. Duplicator maintains the following conditions on the configuration is $\{(x_{i}, y_{i}): 1 \leq i \leq s\}$: 
\begin{enumerate}[label={(A\arabic*)},leftmargin=*]
\item \label{same type cond 1} $x_{i} = u_{1} \Leftrightarrow y_{i} = u_{2}$.

\item \label{same type cond 3} For $1 \leq i_{1} < \cdots < i_{r} \leq s$, call $\big\{x_{i_{1}}, \ldots, x_{i_{r}}\big\}$ an \emph{$x$-cluster} up to round $s$, if they belong to a common principal branch, and no other $x_{j}$ selected so far belongs to it. We analogously define a \emph{$y$-cluster}, up to round $s$. Then $\big\{x_{i_{1}}, \ldots, x_{i_{r}}\big\}$ is an $x$-cluster iff $\big\{y_{i_{1}}, \ldots, y_{i_{r}}\big\}$ is a $y$-cluster, and the principal branches they belong to are of the same depth-$m$ type.

\par Moreover, if $x_{i_{1}}, \ldots, x_{i_{r}} \in S_{\gamma, \ell}$ for some $1 \leq \ell \leq n^{(1)}_{\gamma}$, and $y_{i_{1}}, \ldots, y_{i_{r}} \in T_{\gamma, \ell'}$ for some $1 \leq \ell' \leq n^{(2)}_{\gamma}$, then $\big\{(x_{i_{1}}, y_{i_{1}}), \ldots, (x_{i_{r}}, y_{i_{r}})\big\}$ is winnable for $\left\{\left(S_{\gamma, \ell}, \sigma_{1}\right), \left(T_{\gamma, \ell'}, \sigma_{2}\right), \Sigma, k\right\}.$

\end{enumerate}

\par We first show that Duplicator can maintain these conditions (via strong induction on $s$). Suppose Duplicator has maintained \ref{same type cond 1} and \ref{same type cond 3} up to round $s$. We call a principal branch (in either tree) \emph{free} if no node has been selected, up to round $s$, from it. Otherwise, we call it \emph{occupied}. For any $\gamma \in \Gamma_{m}$, there exists a free principal branch of type $\gamma$ in $T_{1}(u_{1})|_{m}$ iff there exists a free principal branch of type $\gamma$ in $T_{2}(u_{2})|_{m}$. This is evident from \ref{same type cond 3}, \eqref{same type count}, and from the fact that $s \leq k$.
\par Suppose Spoiler, without loss of generality, picks $x_{s+1}$ in round $s+1$. Duplicator's response is split into a few possible cases:
\begin{enumerate}[label={(B\arabic*)},leftmargin=*]
\item \label{same type response 1} If $x_{s+1} = u_{1}$, then Duplicator sets $y_{s+1} = u_{2}$.
\item \label{same type response 2} Suppose $x_{s+1} \in S_{\gamma, \ell}$ for some $\gamma \in \Gamma_{m}$ and $1 \leq \ell \leq n^{(1)}_{\gamma}$, such that $S_{\gamma, \ell}$ is occupied. Let $\big\{x_{i_{1}}, \ldots, x_{i_{r}}\big\}$ be the $x$-cluster up to round $s$ that belongs to $S_{\gamma, \ell}$. By induction hypothesis \ref{same type cond 3}, there exists some $1 \leq \ell' \leq n^{(2)}_{\gamma}$, such that the $y$-cluster $\big\{y_{i_{1}}, \ldots, y_{i_{r}}\big\} \in T_{\gamma, \ell'}$. Moreover $\big\{(x_{i_{1}}, y_{i_{1}}), \ldots, (x_{i_{r}}, y_{i_{r}})\big\}$ is winnable for $\left\{\left(S_{\gamma, \ell}, \sigma_{1}\right), \left(T_{\gamma, \ell'}, \sigma_{2}\right), \Sigma, k\right\}.$ By Definition~\ref{cor node}, Duplicator finds a corresponding node to $x_{s+1}$ in $T_{\gamma, \ell'}$, and sets it to be $y_{s+1}$. 
\par Note that $\big\{(x_{i_{1}}, y_{i_{1}}), \ldots, (x_{i_{r}}, y_{i_{r}}), (x_{s+1}, y_{s+1})\big\}$ is now winnable for $\left\{\left(S_{\gamma, \ell}, \sigma_{1}\right), \left(T_{\gamma, \ell'}, \sigma_{2}\right), \Sigma, k\right\}$, which immediately satisfies \ref{same type cond 3}. 
\item \label{same type response 3} Suppose $x_{s+1} \in S_{\gamma, \ell}$ for some $\gamma \in \Gamma_{m}$ and $1 \leq \ell \leq n^{(1)}_{\gamma}$, such that $S_{\gamma, \ell}$ was free up to round $s$. Duplicator finds an $1 \leq \ell' \leq n^{(2)}_{\gamma}$ such that $T_{\gamma, \ell'}$ was free up to round $s$. By \eqref{ind hyp same type} and Definition~\ref{cor node}, Duplicator finds $y_{s+1}$ in $T_{\gamma, \ell'}$ that is a corresponding node to $x_{s+1}$ with respect to $\DEHR\left[\left(S_{\gamma, \ell}, \sigma_{1}\right), \left(T_{\gamma, \ell'}, \sigma_{2}\right), \Sigma, k\right]$. 
\end{enumerate}

\par It is straightforward to show that Conditions \ref{same type cond 1} and \ref{same type cond 3} imply \ref{DEHR win 2} through \ref{DEHR win 4}. It is also straightforward to show that they imply \ref{DEHR win 1}, when $x_{i}, x_{j}$ (equivalently, by \ref{same type cond 3}, $y_{i}, y_{j}$) belong to the same principal branch, or when $x_{i} = u_{1}$ (equivalently, $y_{i} = u_{2}$ by \ref{same type cond 1}). Suppose now $v_{1}, v'_{1}$ are two distinct children of $u_{1}$ such that $x_{i}$ belongs to the principal branch at $v_{1}$, and $x_{j}$ belongs to that at $v'_{1}$. By \ref{same type cond 3}, there exist distinct children $v_{2}, v'_{2}$ of $u_{2}$ such that $y_{i}$ belongs to the principal branch at $v_{2}$, and $y_{j}$ belongs to that at $v'_{2}$. Moreover, \ref{same type cond 3} implies that $\rho(x_{i}, v_{1}) = \rho(y_{i}, v_{2})$ and $\rho(x_{j}, v'_{1}) = \rho(y_{j}, v'_{2})$. As the distance between $v_{1}$ and $v'_{1}$, as well as that between $v_{2}$ and $v'_{2}$, is $2$, hence $\rho(x_{i}, x_{j}) = \rho(x_{i}, v_{1}) + \rho(x_{j}, v'_{1}) + 2$, and $\rho(y_{i}, y_{j}) = \rho(y_{i}, v_{2}) + \rho(y_{j}, v'_{2}) + 2$, which gives us \ref{DEHR win 1} for $i, j$. 
\end{proof}

\section{The types game is harder than EHR}\label{type to EHR}
The following theorem is the first of the two main steps in proving Theorem~\ref{main 1}. It shows that the types game is in some sense harder for Duplicator to win than the set-pebble Ehrenfeucht game. To this end, fix a positive integer $k$ and the colour set $\Sigma$. Set
\begin{equation}\label{parameters main 2}
D = 4 \cdot 3^{k+2}, \ D_{0} = 25 D, \ M = D/4.
\end{equation}
We have not bothered with optimizing these constants, as showing the existence of $T_{1}$ and $\mathcal{T}_{2}$ as in Theorem~\ref{main 1} is our main objective.

\begin{theorem}\label{main 2}
Let $T_{1}, T_{2}$ be two trees, with roots $\phi_{1}, \phi_{2}$, such that $T_{1}|_{D_{0}/2}$ and $T_{2}|_{D_{0}/2}$ are both isomorphic to a path of length $D_{0}/2$. If Duplicator wins $\Types\left[T_{1}, T_{2}, \overline{\Sigma}, M, k\right]$, where $\overline{\Sigma}$ is defined using the parameters $D$ and $D_{0}$ as in \eqref{parameters main 2}, then she also wins $\EHR\left[T_{1}, T_{2}, \Sigma, k\right]$.
\end{theorem}

The proof itself will contain several parts. Duplicator will maintain Conditions \ref{main 2 cond 1} through \ref{main 2 cond 6}, and her response will be split into several nested cases. The broadest cases are given in subsections. But even before we can go into the analysis of these individual cases, we need some more notation and terminology to make the exposition simpler to read.
\par We note here that whenever we talk about the type of a node in the context of Theorem~\ref{main 2}, it is with respect to the set of augmented colours $\overline{\Sigma}$, depth $M$ and cut-off $k$. So we shall simply refer to the $\left(\overline{\Sigma}, M, k\right)$-type of any node as its type in the proof. The set $\Gamma_{\overline{\Sigma}, M, k}$ is abbreviated as simply $\Gamma$, as here there is no chance of confusion regarding the depth of the types.

\begin{proof}[\textbf{Proof of Theorem~\ref{main 2}}]
According to convention \eqref{convention}, we set $x_{0} = \phi_{1}$ and $y_{0} = \phi_{2}$. We let $P_{1} = \{x_{0} \rightarrow w_{1} \rightarrow \cdots \rightarrow w_{D_{0}/2}\}$ denote the path $T_{1}|_{D_{0}/2}$ and $P_{2} = \{y_{0} \rightarrow z_{1} \rightarrow \cdots \rightarrow z_{D_{0}/2}\}$ denote the path $T_{2}|_{D_{0}/2}$.
\par \textbf{Set round:} In the set round of $\EHR$, Spoiler assigns a $(\Sigma, \col_{0})$-rooted colouring $\sigma_{1}$ to $T_{1}$. Consider the enhanced version $\overline{\sigma_{1}}$ of $\sigma_{1}$. As Duplicator wins $\Types\left[T_{1}, T_{2}, \overline{\Sigma}, M, k\right]$, there exists a $\big(\overline{\Sigma}, (\col_{0}, \col'_{0})\big)$-rooted colouring $\sigma'$ on $T_{2}$, which is her winning reply to $\overline{\sigma_{1}}$. By Lemma~\ref{legal}, $\sigma'$ is legal. Hence there exists some $(\Sigma, \col_{0})$-rooted colouring $\sigma_{2}$ of $T_{2}$, such that $\overline{\sigma_{2}} = \sigma'$. Duplicator then assigns $\sigma_{2}$ to $T_{2}$. This concludes the set round.

\par \textbf{Pebble rounds: } Suppose that in the $i$-th round for some $1 \leq i \leq k$, Spoiler, without loss of generality, chooses $x_{i} \in V(T_{1})$. We shall construct a response $y_{i} \in V(T_{2})$ for Duplicator, via two auxiliary nodes, $u_{i} \in V(T_{1})$ and $v_{i} \in V(T_{2})$. We shall select $u_{i}$ to be a suitable ancestor of $x_{i}$; then select $v_{i}$ so that $u_{i}$ and $v_{i}$ have the same types with respect to $\left(T_{1}, \overline{\sigma_{1}}\right)$ and $\left(T_{2}, \overline{\sigma_{2}}\right)$ respectively. Finally we select $y_{i}$ as a suitable descendant of $v_{i}$. The details of this selection procedure are given, as mentioned earlier, in several different cases. In particular, we set $u_{0} = x_{0}$ and $v_{0} = y_{0}$.

\subsection{Some terminology:} This will help elucidate the explanation of Duplicator's strategy. For any $i, j \in [k]$, we say that $x_{i}$ and $x_{j}$ are \emph{close} if 
\begin{equation}\label{close cond}
\rho(x_{i}, x_{j}) \leq 2 \cdot 3^{k+2-(i \vee j)}, 
\end{equation}
otherwise we call them \emph{far} (we similarly define $y_{i}$ and $y_{j}$ to be close or far). We say that $x_{i}$ and $x_{j}$ \emph{threaten each other} if 
\begin{enumerate}
\item the auxiliary node $u_{i \wedge j}$ does not equal $x_{0}$, 
\item and 
\begin{equation}\label{threat cond}
\left|\Delta_{x}(i, j)\right| \leq 2 \cdot 3^{k+2-(i \vee j)},
\end{equation}
where $\Delta_{x}(i, j) := \left\{d(x_{i}) - d(x_{j})\right\} \bmod D$.
\end{enumerate}
Similarly we say that $y_{i}$ and $y_{j}$ threaten each other if $v_{i \wedge j}$ is not equal to $y_{0}$ and the condition analogous to \eqref{threat cond} holds.

\subsection{Conditions on the configuration Duplicator maintains throughout the pebble rounds:}
Suppose $s$ pebble rounds of the game have been played, where $1 \leq s \leq k$. The following are the conditions Duplicator maintains on the configuration $\{(x_{i}, y_{i}): 1 \leq i \leq s\}$, together with the auxiliary nodes $\{(u_{i}, v_{i}): 1 \leq i \leq s\}$. For all $i, j \in [s]$,
\begin{enumerate}[label={(C \arabic*)},leftmargin=*]
\item \label{main 2 cond 1} $u_{i} = u_{j} \Leftrightarrow v_{i} = v_{j}$. 
\item \label{main 2 cond 2} $u_{i}$ and $v_{i}$ have the same types. In particular, this tells us that $u_{i} = x_{0} \Leftrightarrow v_{i} = y_{0}$, by Remark \ref{po 1} below.
\item \label{main 2 cond 5} The configuration $\left\{(x_{\ell}, y_{\ell}): \ell \in [s], u_{\ell} = u_{i}\right\}$, of pairs that share the auxiliary node $u_{i}$, is winnable for $$\Big\{\big(T_{1}(u_{i})\big|_{M}, \overline{\sigma_{1}}\big), \big(T_{2}(v_{i})\big|_{M}, \overline{\sigma_{2}}\big), \overline{\Sigma}, k\Big\}.$$ Note that from Lemma~\ref{obs imp later}, this gives us $\rho(x_{i}, u_{i}) = \rho(y_{i}, v_{i})$ and $\overline{\sigma_{1}}(x_{i}) = \overline{\sigma_{2}}(y_{i})$.
\item \label{main 2 cond 3} If $u_{i}$ is not $x_{0}$, then 
\begin{equation}\label{main 2 cond 3 ineq}
3^{k+2-i} \leq \rho(x_{i}, u_{i}) \leq M - 3^{k+2-i}.
\end{equation}
If $u_{i}$ equals $x_{0}$, then the upper bound from \eqref{main 2 cond 3 ineq} on $\rho(x_{i}, u_{i})$ holds.
\item \label{main 2 cond 4} $x_{i}$ and $x_{j}$ are close if and only if $y_{i}$ and $y_{j}$ are close as well, and in that case $u_{i} = u_{j}$, $v_{i} = v_{j}$ and $\rho(x_{i}, x_{j}) = \rho(y_{i}, y_{j})$. 
\item \label{main 2 cond 6} If $x_{i}$ and $x_{j}$ threaten each other, then either $u_{i \vee j} = x_{0}$ and $v_{i \vee j} = y_{0}$, or 
\begin{equation} \label{main 2 cond 6 op 2}
d(u_{i}) \equiv d(u_{j}) \bmod D \text{ and } d(v_{i}) \equiv d(v_{j}) \bmod D.
\end{equation}
(\textbf{Note:} If Condition \ref{main 2 cond 2} holds, then $u_{i}, v_{i}$ have the same types, and so do $u_{j}, v_{j}$. Hence $\overline{\sigma_{1}}(u_{i}) = \overline{\sigma_{2}}(v_{i})$, which implies $d(v_{i}) \equiv d(u_{i}) \bmod D$. Similarly, $d(v_{j}) \equiv d(u_{j}) \bmod D$. Hence, in the sequel, we need only verify one of the two congruences in \eqref{main 2 cond 6 op 2}, and the other will follow.)
\end{enumerate}
We prove next that Duplicator can maintain \ref{main 2 cond 1} through \ref{main 2 cond 6} using strong induction on $s$. Without loss of generality, let Spoiler choose $x_{s+1}$ from $T_{1}$ in the $(s+1)$-st pebble round. The response of Duplicator will vary significantly over a few possible cases, which are detailed in the sequel. However, some common remarks apply to many of them, and for brevity, these are mentioned before we go into the nested classifications.

\begin{rem}\label{po 1} 
Since in any $\big(\overline{\Sigma}, (\col_{0}, \col'_{0})\big)$-rooted colouring, the root is the only node which gets the colour $(\col_{0}, \col'_{0})$, its type is unique. Moreover, under the assumptions of Theorem~\ref{main 2}, the type of every $w_{i}$ in $T_{1}$ and every $z_{i}$ in $T_{2}$, where $1 \leq i \leq D_{0}/2$, is unique. This is because each $w_{i}$ (correspondingly $z_{i}$) is the only node such that the second coordinate of $\overline{\sigma_{1}}(w_{i})$ \big(correspondingly $\overline{\sigma_{2}}(z_{i})$\big) is $\col'_{i}$.
\end{rem}

\begin{rem}\label{common 1}
From Condition \ref{main 2 cond 5}, we get $\overline{\sigma_{1}}(x_{j}) = \overline{\sigma_{2}}(y_{j})$, which implies, from \eqref{depth fn}, that 
\begin{enumerate}
\item $d(x_{j}) \leq D_{0} \Leftrightarrow d(y_{j}) \leq D_{0}$, and in that case $d(x_{j}) = d(y_{j})$;
\item $d(x_{j}) > D_{0} \Leftrightarrow d(y_{j}) > D_{0}$, and in that case $d(x_{j}) \equiv d(y_{j}) \bmod D$.
\end{enumerate}
\end{rem}

\begin{rem}\label{po 2} Suppose $x_{i}$ and $x_{j}$ threaten each other, and $i < j$. Then $u_{i}$ does not equal $x_{0}$. From Condition \ref{main 2 cond 2}, $v_{i}$ does not equal $y_{0}$ either. From \ref{main 2 cond 5}, we have $\overline{\sigma_{1}}(x_{i}) = \overline{\sigma_{2}}(y_{i})$, which in turn gives $d(x_{i}) \equiv d(y_{i})\bmod D$. Similarly we have $d(x_{j}) \equiv d(y_{j})\bmod D$. Hence
\begin{align}\label{mod dist same x y}
\bigl|\Delta_{y}(i, j)\bigr| &:= \bigl|\left\{d(y_{i}) - d(y_{j})\right\} \bmod D\bigr|\ = \bigl|\left\{d(x_{i}) - d(x_{j})\right\} \bmod D\bigr| = \bigl|\Delta_{x}(i, j)\bigr|. 
\end{align}
From \eqref{threat cond}, this shows that $y_{i}$ and $y_{j}$ also threaten each other. The converse is also true, i.e. if $y_{i}$ and $y_{j}$ threaten each other, then so do $x_{i}$ and $x_{j}$.
\end{rem}

\begin{rem}\label{super imp rem}
Fix any $\gamma \in \Gamma$, and $1 \leq s \leq k$. If Conditions \ref{main 2 cond 1} and \ref{main 2 cond 2} are maintained throughout the first $s$ pebble rounds of the game, it is not difficult to see that:
\begin{equation}\label{equal number utilized}
\big|\left\{\ell \in [s]: u_{\ell} \text{ of type } \gamma\right\}\big| = \big|\left\{\ell \in [s]: v_{\ell} \text{ of type } \gamma\right\}\big|.
\end{equation}
Now recall that $\overline{\sigma_{2}}$ is a winning reply of Duplicator to $\overline{\sigma_{1}}$ for $\Types\left[T_{1}, T_{2}, \overline{\Sigma}, M,  k\right]$. If $n^{(1)}_{\sigma}$ denotes the number of nodes in $(T_{1}, \overline{\sigma_{1}})$ with type $\gamma$ and $n^{(2)}_{\sigma}$ that in $(T_{2}, \overline{\sigma_{2}})$, then $n^{(1)}_{\sigma} \wedge k = n^{(2)}_{\sigma} \wedge k$. Since $s \leq k$, this shows the following: if $u_{s+1} \neq u_{\ell}$ for all $\ell \in [s]$, and $u_{s+1}$ is of type $\gamma$, then Duplicator can find $v_{s+1}$ of type $\sigma$ in $(T_{2}, \overline{\sigma_{2}})$ such that $v_{s+1} \neq v_{\ell}$ for all $\ell \in [s]$. 
\end{rem}

\subsection{The close move case:} Suppose there exists some $\alpha \in [s]$ such that $x_{s+1}$ is close to $x_{\alpha}$. Duplicator then sets $u_{s+1} = u_{\alpha}$ and $v_{s+1} = v_{\alpha}$. She then selects $y_{s+1}$ as a corresponding node to $x_{s+1}$ in $T_{2}(v_{s+1})\big|_{M}$, given the configuration $\big\{(x_{\ell}, y_{\ell}): \ell \in [s], u_{\ell} = u_{\alpha}\big\}$, which we know is winnable for $\Big\{\big(T_{1}(u_{\alpha})|_{M}, \overline{\sigma_{1}}\big), \big(T_{2}(v_{\alpha})|_{M}, \overline{\sigma_{2}}\big), \overline{\Sigma}, k\Big\}$, by induction hypothesis \ref{main 2 cond 5}. Note that these choices of Duplicator immediately satisfy Condition \ref{main 2 cond 1}. They satisfy Condition \ref{main 2 cond 2} since by induction hypothesis \ref{main 2 cond 2} applied to round $\alpha$, the nodes $u_{\alpha}$ and $v_{\alpha}$ have the same types. 
\par As $y_{s+1}$ is a corresponding node to $x_{s+1}$, by Definition~\ref{cor node}, the configuration $\big\{(x_{\ell}, y_{\ell}): \ell \in [s+1], u_{\ell} = u_{\alpha} = u_{s+1}\big\}$ is now winnable for $\Big\{\big(T_{1}(u_{\alpha})|_{M}, \overline{\sigma_{1}}\big), \big(T_{2}(v_{\alpha})|_{M}, \overline{\sigma_{2}}\big), \overline{\Sigma}, k\Big\}$. This satisfies Condition \ref{main 2 cond 5} for round $s+1$. By induction hypothesis \ref{main 2 cond 3} applied to round $\alpha$, we have $\rho(x_{\alpha}, u_{\alpha}) \leq M - 3^{k+2-\alpha}$. Using triangle inequality and \eqref{close cond} applied to $x_{\alpha}, x_{s+1}$, we get:
\begin{align}
\rho(x_{s+1}, u_{s+1}) &= \rho(x_{s+1}, u_{\alpha}) \leq \rho(x_{s+1}, x_{\alpha}) + \rho(x_{\alpha}, u_{\alpha}) \leq 2 \cdot 3^{k+1-s} + M - 3^{k+2-\alpha} \leq M - 3^{k+1-s}, \nonumber
\end{align}
hence verifying the upper bound for \eqref{main 2 cond 3 ineq}. Further, when $u_{\alpha}$ does not equal $x_{0}$, we have $\rho(x_{\alpha}, u_{\alpha}) \geq 3^{k+2-\alpha}$, hence again applying triangle inequality:
\begin{align}
\rho(x_{s+1}, u_{s+1}) & \geq \rho(x_{\alpha}, u_{\alpha}) - \rho(x_{s+1}, x_{\alpha}) \geq 3^{k+2-\alpha} - 2 \cdot 3^{k+1-s} \geq 3^{k+1-s}, \nonumber
\end{align}
giving us the lower bound in \eqref{main 2 cond 3 ineq}. This completes verification of Condition \ref{main 2 cond 3}.

\par As $\big\{(x_{\ell}, y_{\ell}): \ell \in [s+1], u_{\ell} = u_{\alpha} = u_{s+1}\big\}$ is winnable, by Lemma~\ref{obs imp later}, we get $\rho(x_{s+1}, x_{\alpha}) = \rho(y_{s+1}, y_{\alpha})$, hence $y_{s+1}$ and $y_{\alpha}$ are close as well. Suppose now for some $j \in [s]$, where $j$ is distinct from $\alpha$, the nodes $x_{j}$ and $x_{s+1}$ are close. Then, by \eqref{close cond} applied to the pairs $x_{\alpha}, x_{s+1}$, and $x_{j}, x_{s+1}$, and triangle inequality, we get:
$$\rho(x_{\alpha}, x_{j}) \leq \rho(x_{\alpha}, x_{s+1}) + \rho(x_{s+1}, x_{j}) \leq 4 \cdot 3^{k+1-s} < 2 \cdot 3^{k+2-(\alpha \vee j)}.$$ Hence $x_{j}$ and $x_{\alpha}$ are close as well. By induction hypothesis \ref{main 2 cond 4} applied to round $\alpha \vee j$, we then have $u_{j} = u_{\alpha} = u_{s+1}$ and $v_{j} = v_{\alpha} = v_{s+1}$. Hence, again by Lemma~\ref{obs imp later}, we have $\rho(x_{j}, x_{s+1}) = \rho(y_{j}, y_{s+1})$.

\par Finally we verify that Condition \ref{main 2 cond 6} holds. Suppose there exists some $j \in [s]$ such that $x_{j}$ and $x_{s+1}$ threaten each other. This means that $u_{j}$ does not equal $x_{0}$ and hence $v_{j}$ is not $y_{0}$ either. If it so happens that $u_{s+1} $ equals $x_{0}$ (and hence $v_{s+1}$ equals $y_{0}$), then nothing left to verify. If not, then note that $u_{\alpha} \neq x_{0}$ either. Further,
\begin{align}
\big|\Delta_{x}(j, \alpha)\big| &\leq \bigl|\left\{d(x_{j}) - d(x_{s+1})\right\} \bmod D\bigr| + \bigl|\left\{d(x_{s+1}) - d(x_{\alpha})\right\} \bmod D\bigr| \nonumber\\
&\leq \bigl|\left\{d(x_{j}) - d(x_{s+1})\right\} \bmod D\bigr| + \rho(x_{s+1}, x_{\alpha}) \nonumber\\
&\leq 2 \cdot 3^{k+1-s} + 2 \cdot 3^{k+1-s} < 2 \cdot 3^{k+2-(j \vee \alpha)}. \nonumber
\end{align}
These show that $x_{j}$ and $x_{\alpha}$ also threaten each other. By induction hypothesis \ref{main 2 cond 6} applied to round $j \vee \alpha$, we have $d(u_{j}) \equiv d(u_{\alpha}) \bmod D$ and $d(v_{j}) \equiv d(v_{\alpha}) \bmod D$. Since $u_{s+1} = u_{\alpha}$ and $v_{s+1} = v_{\alpha}$, this gives us Condition \ref{main 2 cond 6} for the indices $j$ and $s+1$.

\subsection{The far move case:} In this case, $x_{s+1}$ is far from $x_{j}$ for every $j \in [s]$. As it's far, the only part of Condition \ref{main 2 cond 4} we need to verify is that, for all $j \in [s]$,
\begin{equation}\label{verify part of cond 4}
\rho(y_{j}, y_{s+1}) > 2 \cdot 3^{k+1-s}.
\end{equation}
This kind of a move calls for splitting into a few possible scenarios, and Duplicator's response in each such scenario will be considerably different. These scenarios are described in the subsections below. 

\subsubsection{\textbf{When the move threatens a previously selected node:}}\label{threat subsection} Suppose there exists some $\alpha \in [s]$ such that $x_{\alpha}$ and $x_{s+1}$ threaten each other. This means that $u_{\alpha}$ does not equal $x_{0}$ and \eqref{threat cond} holds for the pair $x_{\alpha}, x_{s+1}$. This case now splits into a few subcases, as follows (here ``T" represents ``threatening a previous move").\\
\par \textbf{Case T1}: Suppose that there exists an ancestor $u$ of $x_{s+1}$ at distance
\begin{equation}\label{perfect dist}
\rho(x_{s+1}, u) = \rho(x_{\alpha}, u_{\alpha}) + \Delta_{x}(s+1, \alpha),
\end{equation}
from it, and there exists some $\beta \in [s]$ such that $u_{\beta} = u$. Then Duplicator sets $u_{s+1} = u$. 

\par A quick point to observe is that when $d(x_{s+1}) > D_{0}/2$, such an ancestor clearly exists, as is clear from the upper bound of \eqref{main 2 cond 3 ineq} applied to the index $\alpha$, and \eqref{threat cond} applied to the pair $x_{\alpha}, x_{s+1}$. The analysis for this case is now given as follows.

\par Duplicator now sets $v_{s+1} = v_{\beta}$. This is in direct keeping with Condition \ref{main 2 cond 1}, and this condition will therefore not need further verification for \textbf{Case T1}. By induction hypothesis \ref{main 2 cond 2} applied to round $\beta$, we know that $u_{\beta}$ and $v_{\beta}$ have the same types. Hence so do $u_{s+1}$ and $v_{s+1}$, thus showing that Condition \ref{main 2 cond 2} holds for round $s+1$.

\par By induction hypothesis \ref{main 2 cond 5} applied to rounds $s$, we know that $\big\{(x_{\ell}, y_{\ell}): \ell \in [s] \text{ and } u_{\ell} = u_{\beta}\big\}$ is winnable for $\Big\{\big(T_{1}(u_{\beta})|_{M}, \overline{\sigma_{1}}\big), \big(T_{2}(v_{\beta})|_{M}, \overline{\sigma_{2}}\big), \overline{\Sigma}, k\Big\}$. Hence, by Definition~\ref{cor node}, Duplicator can find a corresponding node to $x_{s+1}$ in $T_{2}(v_{s+1})\big|_{M}$, which she then sets to be $y_{s+1}$. 

\par Note that, also by Definition~\ref{cor node}, the new configuration $\big\{(x_{\ell}, y_{\ell}): \ell \in [s+1], u_{\ell} = u_{s+1}\big\}$ is winnable for $\Big\{\big(T_{1}(u_{s+1})|_{M}, \overline{\sigma_{1}}\big), \big(T_{2}(v_{s+1})|_{M}, \overline{\sigma_{2}}\big), \overline{\Sigma}, k\Big\}$. This directly gives us validity of Condition \ref{main 2 cond 5} up to and including step $s+1$. 

\par In verifying Condition \ref{main 2 cond 3} for \textbf{Case T1}, it is important to note that we only make use of \eqref{perfect dist}, and bounds on $\Delta_{x}(\alpha, s+1)$ and $\rho(x_{\alpha}, u_{\alpha})$. This is worth noting, because this exact same verification will go through for \textbf{Case T2}. We now show that the upper bound of \eqref{main 2 cond 3 ineq} holds for $s+1$ (using induction hypothesis \ref{main 2 cond 3} applied to round $\alpha$ and \eqref{threat cond}):
\begin{align}
\rho(x_{s+1}, u_{s+1}) &= \Delta_{x}(s+1, \alpha) + \rho(x_{\alpha}, u_{\alpha}) \nonumber\\
&\leq 2 \cdot 3^{k+1-s} + M - 3^{k+2-\alpha} \leq M - 3^{k+1-s}, \text{ as } \alpha \leq s. \nonumber
\end{align}
Note that by \eqref{threat cond}, we also have $\Delta_{x}(s+1, \alpha) \geq -2 \cdot 3^{k+1-s}$. By definition, as $x_{\alpha}$ and $x_{s+1}$ threaten each other, hence $u_{\alpha}$ does not equal $x_{0}$, hence the lower bound of \eqref{main 2 cond 3 ineq} holds for the index $\alpha$. If $u_{s+1}$ does not equal $x_{0}$, we have
\begin{align}
\rho(x_{s+1}, u_{s+1}) &= \Delta_{x}(s+1, \alpha) + \rho(x_{\alpha}, u_{\alpha}) \nonumber\\
&\geq -2 \cdot 3^{k+1-s} + 3^{k+2-\alpha} \geq 3^{k+1-s}. \nonumber
\end{align}
This completes the verification for \eqref{main 2 cond 3 ineq} and hence Condition \ref{main 2 cond 3} for round $s+1$. 
\par Once again, in verifying Condition \ref{main 2 cond 6} for \textbf{Case T1}, it is important to note that we crucially make use of \eqref{perfect dist} and \eqref{threat cond} applied to suitable pairs; hence this verification too goes through verbatim for \textbf{Case T2}. We first verify Condition \ref{main 2 cond 6} for the pair $\alpha, s+1$. If $u_{s+1} = x_{0}$ (which implies $v_{s+1} = y_{0}$), nothing left to verify. So, assume $u_{s+1}$ does not equal $x_{0}$. Then
\begin{align} \label{verify cond 6 1}
d(u_{s+1}) &= d(x_{s+1}) - \rho(x_{s+1}, u_{s+1}) \nonumber\\
&= d(x_{s+1}) - \rho(x_{\alpha}, u_{\alpha}) - \Delta_{x}(s+1, \alpha) \nonumber\\
&= d(x_{s+1}) - \rho(x_{\alpha}, u_{\alpha}) - \big\{d(x_{s+1}) - d(x_{\alpha})\big\} \bmod D \nonumber\\
&\equiv \big\{d(x_{\alpha}) - \rho(x_{\alpha}, u_{\alpha})\big\}\bmod D \equiv d(u_{\alpha}) \bmod D. 
\end{align}
This completes the verification of \eqref{main 2 cond 6 op 2} applied to the pair $\alpha, s+1$. Suppose now $j$, distinct from $\alpha$, is another index in $[s]$ such that $x_{j}$ and $x_{s+1}$ threaten each other. This means that $u_{j}$ does not equal $x_{0}$ and \eqref{threat cond} holds for the pair $j, s+1$. We also have $u_{\alpha} \neq x_{0}$. By triangle inequality and \eqref{threat cond} applied to pairs $\alpha, s+1$ and $j, s+1$, we get:
\begin{align}
\big|\Delta_{x}(j, \alpha)\big| &= \big|\left\{d(x_{j}) - d(x_{\alpha})\right\} \bmod D\big| \nonumber\\
&\leq \big|\left\{d(x_{j}) - d(x_{s+1})\right\} \bmod D\big| + \big|\left\{d(x_{s+1}) - d(x_{\alpha})\right\} \bmod D\big| \nonumber\\
&\leq 4 \cdot 3^{k+1-s} < 2 \cdot 3^{k+2-(j \vee \alpha)}. \nonumber
\end{align}
This shows that $x_{j}$ and $x_{\alpha}$ also threaten each other. Hence, by induction hypothesis \ref{main 2 cond 6} applied to round $j \vee \alpha$, we must have $d(u_{j}) \equiv d(u_{\alpha}) \bmod D$. Combining this with \eqref{verify cond 6 1}, we get $d(u_{s+1}) \equiv d(u_{j}) \bmod D$. This completes the verification of the entire Condition \ref{main 2 cond 6}.
\par Finally we verify \eqref{verify part of cond 4}. For this we shall consider various categories of $j \in [s]$. If $j$ is such that $u_{j} = u_{s+1}$, then from Condition \ref{main 2 cond 5}, and Lemma~\ref{obs imp later}, we conclude that $\rho(y_{j}, y_{s+1}) = \rho(x_{j}, x_{s+1})$, which implies \eqref{verify part of cond 4}. The next few categorizations of $j$ are actually relevant also in \textbf{Case T2}, and the verification of \eqref{verify part of cond 4} for \textbf{Case T2} follows exactly the steps described below.
\par Suppose $j$ is such that $x_{j}$ and $x_{s+1}$ threaten each other. This means that $u_{j}$ is not $x_{0}$. From \eqref{main 2 cond 6 op 2}, we have two possibilities:
\par \textbf{Possibility 1:} Here, $d(v_{j}) = d(v_{s+1})$. Our only interest is when $v_{j} \neq v_{s+1}$ (since otherwise, $u_{j} = u_{s+1}$, and that has been dealt with above). But this means that $\rho(v_{j}, v_{s+1}) \geq 2$. We have already verified Condition \ref{main 2 cond 3}, hence $\rho(y_{s+1}, v_{s+1}) = \rho(x_{s+1}, u_{s+1}) \geq 3^{k+1-s}$. As $u_{j}$ is not $x_{0}$, by induction hypothesis \ref{main 2 cond 3} applied to round $j$, we get $\rho(y_{j}, v_{j}) = \rho(x_{j}, u_{j}) \geq 3^{k+2-j}$. We therefore have:
\begin{align}
\rho(y_{j}, y_{s+1}) &= \rho(y_{j}, v_{j}) + \rho(v_{j}, v_{s+1}) + \rho(v_{s+1}, y_{s+1}) \nonumber\\
&\geq 3^{k+2-j} + 2 + 3^{k+1-s} > 2 \cdot 3^{k+1-s}. \nonumber
\end{align}

\par \textbf{Possibility 2:} Here $d(v_{j}) \neq d(v_{s+1})$, but then by \eqref{main 2 cond 6 op 2}, we must have $\bigl|d(v_{j}) - d(v_{s+1})\bigr| \geq D$. We now use the upper bounds from \eqref{main 2 cond 3 ineq}, applied to both $j$ and $s+1$, and triangle inequality, to get:
\begin{align}
\rho(y_{j}, y_{s+1}) &\geq \rho(v_{j}, v_{s+1}) - \rho(v_{j}, y_{j}) - \rho(v_{s+1}, y_{s+1}) \nonumber\\
&\geq \bigl|d(v_{j}) - d(v_{s+1})\bigr| - \rho(u_{j}, x_{j}) - \rho(u_{s+1}, x_{s+1}) \nonumber\\
&\geq D - M + 3^{k+2-j} - M + 3^{k+1-s} \nonumber\\
&= 4 \cdot 3^{k+2} - 2 \cdot 3^{k+2} + 3^{k+2-j} + 3^{k+1-s} \nonumber\\
&= 2 \cdot 3^{k+2} + 3^{k+2-j} + 3^{k+1-s} > 2 \cdot 3^{k+1-s}. \nonumber
\end{align}
Finally, consider $j \in [s]$ such that $x_{j}$ and $x_{s+1}$ do not threaten each other. Again there are two possibilities:
\par \textbf{Possibility 1:} $x_{j}$ and $x_{s+1}$ do not threaten each other because $u_{j}$ equals $x_{0}$. Then note that by induction hypothesis \ref{main 2 cond 2}, we know that $v_{j}$ equals $y_{0}$, and by induction hypothesis \ref{main 2 cond 3} applied to round $j$, we have:
\begin{align}
d(y_{j}) &= d(v_{j}) + \rho(v_{j}, y_{j}) = d(y_{0}) + \rho(u_{j}, x_{j}) \leq M - 3^{k+2-j} < D_{0}/2. \nonumber
\end{align}
This means that $y_{j}$ lies on the path $T_{1}|_{D_{0}/2}$. If we have $d(y_{s+1}) > D_{0}/2$, then $$\rho(y_{s+1}, y_{j}) \geq d(y_{s+1}) - d(y_{j}) \geq D_{0}/2 - M + 3^{k+2-j} > 2 \cdot 3^{k+1-s}.$$ If on the other hand, we have $d(y_{s+1}) \leq D_{0}/2$, then since both $y_{s+1}$ and $y_{j}$ lie on the path $T_{1}|_{D_{0}/2}$, by Remark \ref{common 1}, we have 
$$\rho(y_{s+1}, y_{j}) = \left|d(y_{s+1}) - d(y_{j})\right| = \left|d(x_{s+1}) - d(x_{j})\right| = \rho(x_{j}, x_{s+1}),$$
hence again the desired inequality \eqref{verify part of cond 4} holds.
\par \textbf{Possibility 2:} If $u_{j}$ does not equal $x_{0}$, then $x_{j}$ and $x_{s+1}$ do not threaten each other because \eqref{threat cond} does not hold. From \eqref{mod dist same x y}, we then have $\rho(y_{j}, y_{s+1}) \geq \bigl|\Delta_{y}(j, s+1)\bigr| > 2 \cdot 3^{k+1-s}$. This completes the verification of \eqref{verify part of cond 4}.
\par This concludes the verification of Conditions \ref{main 2 cond 1} through \ref{main 2 cond 6} for \textbf{Case T1}, up to and including step $s+1$.\\

\par \textbf{Case T2}: Suppose there exists an ancestor $u_{s+1}$ of $x_{s+1}$ such that \eqref{perfect dist} holds, but there exists no $\ell \in [s]$ with $u_{\ell} = u_{s+1}$. By Remark \ref{super imp rem}, Duplicator can find a $v_{s+1}$ with the same type as $u_{s+1}$, and $v_{s+1} \neq v_{\ell}$ for all $\ell \in [s]$. This also immediately gives us validity of Condition \ref{main 2 cond 1} and \ref{main 2 cond 2} up to round $(s+1)$.

\par By Lemma~\ref{same type win DEHR} and Duplicator's choice of $v_{s+1}$, note that she wins $\DEHR\Big[\big(T_{1}(u_{s+1})|_{M}, \overline{\sigma_{1}}\big), \big(T_{2}(v_{s+1})|_{M}, \overline{\sigma_{2}}\big), \overline{\Sigma}, k\Big]$. Hence by Definition~\ref{cor node}, she can find $y_{s+1}$ in $T_{2}(v_{s+1})\big|_{M}$ which is a corresponding node to $x_{s+1}$, and this is her choice for round $s+1$.

\par By Definition~\ref{cor node}, this tells us that $\big\{(x_{s+1}, y_{s+1})\big\}$ is a winnable configuration for $\Big\{\big(T_{1}(u_{s+1})|_{M}, \overline{\sigma_{1}}\big), \big(T_{2}(v_{s+1})|_{M}, \overline{\sigma_{2}}\big), \overline{\Sigma}, k\Big\}$, hence Condition \ref{main 2 cond 5} holds up to round $s+1$. As mentioned above, the verification of both Conditions \ref{main 2 cond 3} and \ref{main 2 cond 6} are done in exactly the same way as in \textbf{Case T1}, since they make use of the fact that $\rho(u_{s+1}, x_{s+1})$ satisfies \eqref{perfect dist}. The verification of \eqref{verify part of cond 4} in \textbf{Case T2} is a subset of how we verify it for \textbf{Case T1} (also mentioned above), since we do not have any $j \in [s]$ where $u_{j} = u_{s+1}$; the rest is exactly the same as before. 

\par This concludes the verification of Conditions \ref{main 2 cond 1} through \ref{main 2 cond 6} for \textbf{Case T1}, up to and including step $s+1$.\\

\par \textbf{Case T3}: There does not exist any ancestor of $x_{s+1}$ such that \eqref{perfect dist} holds. This can happen only if $d(x_{s+1}) < \Delta_{x}(\alpha, s+1) + \rho(x_{\alpha}, u_{\alpha})$. In particular, from \eqref{threat cond}, and by induction hypothesis \ref{main 2 cond 3} applied to round $\alpha$, we have 
\begin{align}\label{to verify cond 3 later}
d(x_{s+1}) & \leq M - 3^{k+2-\alpha} + 2 \cdot 3^{k+1-s} \leq M - 3^{k+1-s}, 
\end{align}
thus showing that $x_{s+1}$ lies on the path $T_{1}|_{M} \subseteq T_{1}|_{D_{0}/2}$. In this case, she sets $u_{s+1} = x_{0}$ and $v_{s+1} = y_{0}$, which automatically give us validation of Conditions \ref{main 2 cond 1} and \ref{main 2 cond 2}.
\par By induction hypothesis \ref{main 2 cond 5}, we know that $\big\{(x_{\ell}, y_{\ell}): \ell \in [s], u_{\ell} = u_{0} = x_{0}\big\}$ is winnable for $\Big\{\big(T_{1}|_{M}, \overline{\sigma_{1}}\big), \big(T_{2}|_{M}, \overline{\sigma_{2}}\big), \overline{\Sigma}, k\Big\}$. Hence by Definition~\ref{cor node}, Duplicator chooses $y_{s+1}$ to be a corresponding node to $x_{s+1}$ in $T_{2}|_{M}$. We wish to stress here the fact that the choice of the corresponding node is actually unique. This is because, $x_{s+1}$ lies on the path $T_{1}|_{D_{0}/2}$. If $x_{s+1} = w_{i}$ for some $i \leq M - 3^{k+1-s}$, then the only node in $T_{2}|_{D_{0}/2}$ with the same type as $w_{i}$ is $z_{i}$ (see Remark \ref{po 1}). Further, this choice of $y_{s+1}$, by Definition~\ref{cor node}, guarantees that $\big\{(x_{\ell}, y_{\ell}): \ell \in [s+1], u_{\ell} = u_{0} = x_{0}\big\}$ is winnable for $\Big\{\big(T_{1}|_{M}, \overline{\sigma_{1}}\big), \big(T_{2}|_{M}, \overline{\sigma_{2}}\big), \overline{\Sigma}, k\Big\}$, hence giving us validity of Condition \ref{main 2 cond 5} for round $s+1$.

\par From \eqref{to verify cond 3 later} we also get the upper bound of \eqref{main 2 cond 3 ineq} for round $s+1$. We do not need to verify the lower bound because here $u_{s+1} = x_{0}$. We need no verification for Condition \ref{main 2 cond 6} since we already have $u_{s+1} = x_{0}$ and $v_{s+1} = y_{0}$. Finally, we come to the verification of \eqref{verify part of cond 4}. If $j$ is such that $x_{j}$ lies on the path $T_{1}|_{D_{0}/2}$ and equals $w_{i'}$, then by Remark \ref{po 1}, $y_{j} = z_{i'}$, for any $i' \leq D_{0}/2$. In this case $\rho(y_{j}, y_{s+1}) = |i - i'| = \rho(x_{j}, x_{s+1})$, thus giving us \eqref{verify part of cond 4}. If $j$ is such that $d(x_{j}) > D_{0}/2$, then by Remark \ref{common 1} we know that $d(y_{j}) > D_{0}/2$ as well. Hence 
$$\rho(y_{j}, y_{s+1}) = d(y_{j}) - d(y_{s+1}) > \frac{D_{0}}{2} - M + 3^{k+1-s} > 2 \cdot 3^{k+1-s}.$$
We thus conclude the verifications of Conditions \ref{main 2 cond 1} through \ref{main 2 cond 6} for \textbf{Case T3} up to and including round $s+1$.

\subsubsection{\textbf{When the move threatens no previously selected node:}} Here Duplicator chooses $u_{s+1}$ to be the ancestor at distance $3^{k+1-s}$ from $x_{s+1}$. Such an ancestor exists since $x_{s+1}$ is far from every previously selected $x_{j}$, which implies that we have $\rho(x_{0}, x_{s+1}) > 2 \cdot 3^{k+1-s}$. Note that $\rho(x_{s+1}, u_{s+1}) = 3^{k+1-s}$ immediately gives us \eqref{main 2 cond 3 ineq} and hence Condition \ref{main 2 cond 3}. So we do not verify this separately for the two subcases anymore.
\par As $x_{s+1}$ does not threaten any $x_{j}$ for $j \in [s]$, hence we do not need to verify Condition \ref{main 2 cond 6} henceforth. We can do a common verification of \eqref{verify part of cond 4} right here for both \textbf{Case NT1} and \textbf{Case NT2}, subject to the condition that we verify Conditions \ref{main 2 cond 2} and \ref{main 2 cond 5} separately for each of them. For any $j \in [s]$, since $x_{j}$ and $x_{s+1}$ do not threaten each other, by Remark \ref{po 2}, neither do $y_{j}$ and $y_{s+1}$. This can happen because of two reasons:
\begin{enumerate}
\item This happens because $u_{j} = x_{0}$. By induction hypothesis \ref{main 2 cond 2} applied to round $j$, this also gives $v_{j} = y_{0}$. Now, from induction hypothesis \ref{main 2 cond 3} applied to round $j$, we have $d(y_{j}) = \rho(y_{j}, v_{j}) = \rho(y_{j}, y_{0}) \leq M - 3^{k+1-j}$. Hence $y_{j}$ lies on the path $T_{2}|_{D_{0}/2}$.  
\par If $y_{s+1}$ also lies on $T_{2}|_{D_{0}/2}$, then from Remark \ref{common 1}, we have $\rho(y_{j}, y_{s+1}) = \left|d(y_{j}) - d(y_{s+1})\right| = \left|d(x_{j}) - d(x_{s+1})\right| = \rho(x_{j}, x_{s+1})$. Hence \eqref{verify part of cond 4} holds.
\par If $y_{s+1}$ does not lie on $T_{2}|_{D_{0}/2}$, then we have $\rho(y_{j}, y_{s+1}) = d(y_{s+1}) - d(y_{j}) > D_{0}/2 - M + 3^{k+1-j} > 2 \cdot 3^{k+1-s}$.

\item $x_{j}$ and $x_{s+1}$ do not threaten each other because \eqref{threat cond} does not hold for them. By \eqref{mod dist same x y}, we have $\rho(y_{j}, y_{s+1}) \geq \left|\Delta_{y}(j, s+1)\right| = \left|\Delta_{x}(j, s+1)\right| > 2 \cdot 3^{k+1-s}$.
\end{enumerate}
This completes the verification of \eqref{verify part of cond 4}, subject to the condition that we verify \ref{main 2 cond 2} and \ref{main 2 cond 3}.

\par We now go into the detailed analysis of the subcases, as follows (here ``NT" stands for ``not threatening").\\

\par \textbf{Case NT1}: There exists some $\beta \in [s]$ such that $u_{\beta} = u_{s+1}$. Then Duplicator selects $v_{s+1} = v_{\beta}$, which is in direct keeping with Condition \ref{main 2 cond 1}. So this condition does not require further verification for \textbf{Case NT1}. By induction hypothesis \ref{main 2 cond 2} applied to round $\beta$, the nodes $u_{\beta}$ and $v_{\beta}$ have the same types, and hence so do $u_{s+1}$ and $v_{s+1}$. Thus Condition \ref{main 2 cond 2} holds for round $s+1$.

\par By induction hypothesis, the configuration $\big\{(x_{\ell}, y_{\ell}): \ell \in [s] \text{ and } u_{\ell} = u_{\beta}\big\}$, which is the set of previously chosen pairs that share the auxiliary node $u_{\beta}$, is winnable for $\Big\{\big(T_{1}\left(u_{\beta}\right)|_{M}, \overline{\sigma_{1}}\big), \big(T_{2}\left(v_{\beta}\right)|_{M}, \overline{\sigma_{2}}\big), \overline{\Sigma}, k\Big\}$. By Definition~\ref{cor node}, Duplicator can therefore select a corresponding node to $x_{s+1}$, in $T_{2}\left(v_{\beta}\right)\big|_{M}$, and set that to be $y_{s+1}$. Again by Definition~\ref{cor node}, this choice makes $\big\{(x_{\ell}, y_{\ell}): \ell \in [s+1] \text{ and } u_{\ell} = u_{s+1}\big\}$ a winnable configuration for $\Big\{\big(T_{1}\left(u_{s+1}\right)|_{M}, \overline{\sigma_{1}}\big), \big(T_{2}\left(v_{s+1}\right)|_{M}, \overline{\sigma_{2}}\big), \overline{\Sigma}, k\Big\}$. Thus it validates Condition \ref{main 2 cond 5} up to and including round $s+1$. \\

\par \textbf{Case NT2}: There exists no $\ell \in [s]$ such that $u_{\ell} = u_{s+1}$. By Remark \ref{super imp rem}, Duplicator can find $v_{s+1}$ in $T_{2}$ such that $v_{\ell} \neq v_{s+1}$ for all $\ell \in [s]$, and $u_{s+1}$ and $v_{s+1}$ have the same types. This choice immediately allows Conditions \ref{main 2 cond 1} and \ref{main 2 cond 2} to hold. By Lemma~\ref{same type win DEHR} and Definition~\ref{cor node}, she now selects $y_{s+1}$ as a corresponding node to $x_{s+1}$, in $T_{2}(v_{s+1})\big|_{M}$. This makes $\{(x_{s+1}, y_{s+1})\}$ a winnable configuration for $\Big\{\big(T_{1}(u_{s+1})|_{M}, \overline{\sigma_{1}}\big), \big(T_{2}(v_{s+1})|_{M}, \overline{\sigma_{2}}\big), \overline{\Sigma}, k\Big\}$, hence validating Condition \ref{main 2 cond 5}.\\

\par All the possible scenarios for Duplicator's response have now been analyzed, and we have come to the end of the inductive proof that indeed Duplicator can always maintain Conditions \ref{main 2 cond 1} through \ref{main 2 cond 6}. It is straightforward to verify that these conditions are stricter than what she needs to win the set-pebble game, i.e.\ these conditions imply Conditions \ref{EHR 1} through \ref{EHR 3}. We still provide here a quick explanation of this. For $i, j \in [s]$:
\begin{enumerate}
\item Verifying \ref{EHR 1}: If $\pi(x_{j}) = x_{i}$, then $\rho(x_{i}, x_{j}) = 1 < 3^{k+2-(i \vee j)}$, therefore these are close. From Condition \ref{main 2 cond 4}, we know that $u_{i} = u_{j}$ and $v_{i} = v_{j}$; from Condition \ref{main 2 cond 3}, we know that $\big\{(x_{i}, y_{i}), (x_{j}, y_{j})\big\}$ is winnable for $\Big\{\big(T_{1}(u_{i})|_{M}, \overline{\sigma_{1}}\big), \big(T_{2}(v_{i})|_{M}, \overline{\sigma_{2}}\big), \overline{\Sigma}, k\Big\}$. Then Condition \ref{DEHR win 2} for winning the DEHR on these two trees gives us $\pi(y_{j}) = y_{i}$.
\item Verifying \ref{EHR 2}: Immediate from Condition \ref{main 2 cond 5}.
\item Verifying \ref{EHR 3}: Immediate again from Condition \ref{main 2 cond 4}.
\end{enumerate}
This concludes the proof of Theorem \ref{main 2}.
\end{proof}

\section{Duplicator wins the types game with positive probability}\label{type win positive probab}
The final key to the proof of Theorem~\ref{main 1}, now that we have Theorem~\ref{main 2}, is the following theorem.
\begin{theorem}\label{main 3}
Fix any set of colours $\Sigma$ with a distinguished root colour $\col_{0}$, and positive integers $m$, $k$ and $L$. There exists a finite tree $T_{1}$ and a family $\mathcal{T}_{2}$ of infinite trees, such that 
\begin{enumerate}
\item the subtree $T_{1}|_{L}$ is isomorphic to a path of length $L$, and so is $T_{2}|_{L}$ for every $T_{2} \in \mathcal{T}_{2}$,
\item Duplicator wins $\Types\left[T_{1}, T_{2}, \Sigma, m, k\right]$ for every $T_{2} \in \mathcal{T}_{2}$.
\end{enumerate}
\end{theorem}
Note that when we apply it to prove Theorem~\ref{main 1}, we use the relevant colour set, and parameters mentioned in \eqref{parameters main 2}. 


\par Recall that $T(v)$ denotes the subtree at the node $v$ in $V(T)$. Given a colouring $\sigma: V(T) \rightarrow \Sigma$, we shall, as before, abuse notation slightly and consider the coloured tree $\left(T(v), \sigma\right)$, where $\sigma$ automatically means the restriction of the colouring to $T(v)$. Further, as before, we denote its truncation consisting of nodes at generations at most $n$, along with the assignment $\sigma$ restricted to this truncation, by $\left(T(v)|_{n}, \sigma\right)$. 

\par All the definitions that follow are in terms of the colour set $\Sigma$ containing the special root colour $\col_{0}$, depth-parameter $m$ and cutoff $k$.
\begin{defn}\label{deficient}
For any set of types $S \subseteq \Gamma$, we call $S$ \emph{unavoidable} for a tree $T$ if for every $(\Sigma, \col_{0})$-rooted colouring $\sigma$ of $T$, there exists some $v \in V(T)$ such that $\Type(v) = \Type _{\Sigma, m, k, (T, \sigma)}(v)$ belongs to $S$. 
\par We define $S$ to be \emph{deficient} for a tree $T$ if $S^{c}$ is unavoidable for $T$. In other words, if $S$ is deficient for $T$, then there exists no $(\Sigma, \col_{0})$-rooted colouring $\sigma$ of $T$ such that
\begin{equation}\label{wanting eq}
\left\{\Type(u): u \in V(T)\right\} \subseteq S,
\end{equation}
where $\Type(u) = \Type_{\Sigma, m, k, (T, \sigma)}$ is the type of $v$ with respect to $(T, \sigma)$.
\end{defn}

\par The following lemma is crucial in the proof of Theorem~\ref{main 3}, and relies on a compactness argument.
\begin{lemma}\label{finite tree noncolourable}
Suppose a set of types $S \subseteq \Gamma$ is deficient for some tree. Then it is deficient for some \emph{finite} tree.
\end{lemma}

\begin{proof}
We prove this via contradiction and a compactness argument. Suppose the claim of the lemma is false. Then for every $n \in \mathbb{N}$, we must have $S$ \emph{not} deficient for $T|_{n}$. This means that we can find a $(\Sigma, \col_{0})$-rooted colouring $\sigma_{n}$ of $T|_{n}$ such that \eqref{wanting eq} holds for $\sigma = \sigma_{n}$ and the subtree $T|_{n}$. 
\par Name the nodes of the tree $T$ in a breadth-first manner (with siblings labeled in a lexicographic order) as $\phi = v_{0}, v_{1}, v_{2} \ldots$. Since $\Sigma$ is a finite set, hence we can find some colour $c_{1}$ in $\Sigma$ and an infinite sequence $N_{1} = \left\{n^{(1)}_{k}: k \in \mathbb{N}\right\}$ such that $\sigma_{n^{(1)}_{k}}(v_{1}) = c_{1}$ for all $k$. Next, we can find a colour $c_{2}$ and an infinite subsequence $N_{2} = \left\{n^{(2)}_{k}: k \in \mathbb{N}\right\}$ of $N_{1}$, such that $\sigma_{n^{(2)}_{k}}(v_{2}) = c_{2}$ for all $k$.
\par Continuing like this, for every $i \in \mathbb{N}$, we can find colour $c_{i}$ and an infinite subsequence $N_{i} = \left\{n^{(i)}_{k}: k \in \mathbb{N}\right\}$ with $N_{i} \subseteq N_{i-1}$, such that $\sigma_{n^{(i)}_{k}}(v_{i}) = c_{i}$ for all $k$. Now, consider the diagonal subsequence $N = \left\{n^{(k)}_{k}: k \in \mathbb{N}\right\}$. By our construction of the sequences, we have, for every $i \geq 1$, 
$$\sigma_{n^{(k)}_{k}}(v_{i}) = c_{i} \text{ for all } k \geq i.$$
Consider now the following $(\Sigma, \col_{0})$-rooted colouring $\sigma$ of $T$, defined by $\sigma(v_{i}) = c_{i}$ for all $i \geq 1$. Then for every $i \in \mathbb{N}$, for all sufficiently large $k$, note that
\begin{equation}
\Type(v_{i}) \text{ with respect to } \left(T, \sigma\right) \text{ concides with } \Type(v_{i}) \text{ with respect to } \left(T\Big|_{n^{(k)}_{k}}, \sigma_{n^{(k)}_{k}}\right),
\end{equation}
and the latter must belong to $S$, since by our assumption, $\sigma_{n^{(k)}_{k}}$ satisfies \eqref{wanting eq} with respect to $S$. This shows that $\sigma$ is a $(\Sigma, \col_{0})$-rooted colouring of the entire $T$ such that all the resulting types are contained in $S$, thus contradicting the hypothesis of the lemma. Hence we conclude that indeed, there exists some $n \in \mathbb{N}$ such that $S$ is deficient for $T|_{n}$.
\end{proof}
If a subset $S$ is deficient for some tree, choose and fix, by Lemma~\ref{finite tree noncolourable}, a finite $T(S)$ such that $S$ is deficient for $T(S)$.
\begin{defn}
A subset $S$ of $\Gamma$ is defined to be \emph{adequate} if it is not deficient for any tree in $\mathcal{T}$. In other words, given any tree $T \in \mathcal{T}$, there exists at least one $(\Sigma, \col_{0})$-rooted colouring $\sigma$ of $T$ such that the types of all the nodes in $(T, \sigma)$ are contained in $S$.
\end{defn}
Set
\begin{equation}\label{not ubi}
Q = Q(\Sigma, m, k) = \left\{S \subseteq \Gamma: S \text{ is not adequate}\right\},
\end{equation}
or, in other words, a set of types belongs to $Q$ if and only if it is deficient for some tree. Since $\Gamma$ is finite, so is $Q$. For any finite tree $T \in \mathcal{T}$, let $N(T)$ denote the number of possible colourings to $T$ where the colour $\col_{0}$ is not used \Big(in other words, $N(T) = \left\{|\Sigma|-1\right\}^{|V(T)|}$\Big). 

\subsection{The construction of the two trees} \label{construction}
This subsection is concerned with the construction of the finite tree $T_{1}$ and the family $\mathcal{T}_{2}$ of infinite trees, as mentioned in Theorem~\ref{main 3} above. For arbitrary $L$, we let $T_{1}|_{L}$ be isomorphic to a path of length $L$. Let $T_{1}|_{L} = \left\{\phi_{1} \rightarrow w_{1} \rightarrow \cdots \rightarrow w_{L}\right\}$. We let $w_{L}$ have children $u_{S,1}, \ldots, u_{S, k \cdot N(T(S))}$ for every $S \in Q$. For each $1 \leq i \leq k \cdot N(T(S))$, we have $T_{1}\left(u_{S, i}\right) = T(S)$, i.e.\ we hang a copy of $T(S)$ from every one of $u_{S,1}, \ldots, u_{S, k \cdot N(T(S))}$. This completes the description of $T_{1}$.

\par The trees in the family $\mathcal{T}_{2}$ are constructed as follows. Consider \emph{any} infinite tree $t_{2}$ (i.e.\ any deterministic infinite tree). We construct a tree $T_{2}$ in the family $\mathcal{T}_{2}$ corresponding to $t_{2}$ as follows: we let $T_{2}|_{L}$ again be a path of length $L$. Let $T_{2}|_{L} = \left\{\phi_{2} \rightarrow z_{1} \rightarrow \cdots \rightarrow z_{L}\right\}$. For every $S \in Q$, we let $z_{L}$ have children $v_{S,1}, \ldots, v_{S, k \cdot N(T(S))}$, such that for each $1 \leq i \leq k \cdot N(T(S))$, the subtree $T_{2}\left(v_{S, i}\right) = T(S)$. We also let $z_{L}$ have an additional child $v$ such that $T_{2}(v) \cong t_{2}$. The family $\mathcal{T}_{2}$ is the collection of $T_{2}$ corresponding to $t_{2}$ for all possible infinite trees $t_{2}$. 

\begin{rem}\label{positive probab}
Clearly, if $\mu$ is a probability measure on $\mathcal{T}$ satisfying the condition given in Theorem \ref{main}, then we have $\mu[\mathcal{T}_{2}] > 0$. 
\end{rem}

\par Rough images of $T_{1}$ and $T_{2}$ corresponding to $t_{2}$ are shown in Figures~\ref{finite_image} and \ref{infinite_image} \big(only the copies of $T(S_{1})$ and $T(S_{2})$ are shown as an example for $S_{1}, S_{2} \in Q$, but it is implicitly understood that copies of $T(S)$ for every other $S \in Q$ are present\big).
\begin{figure}[h!]
  \centering
    \includegraphics[width=0.7\textwidth]{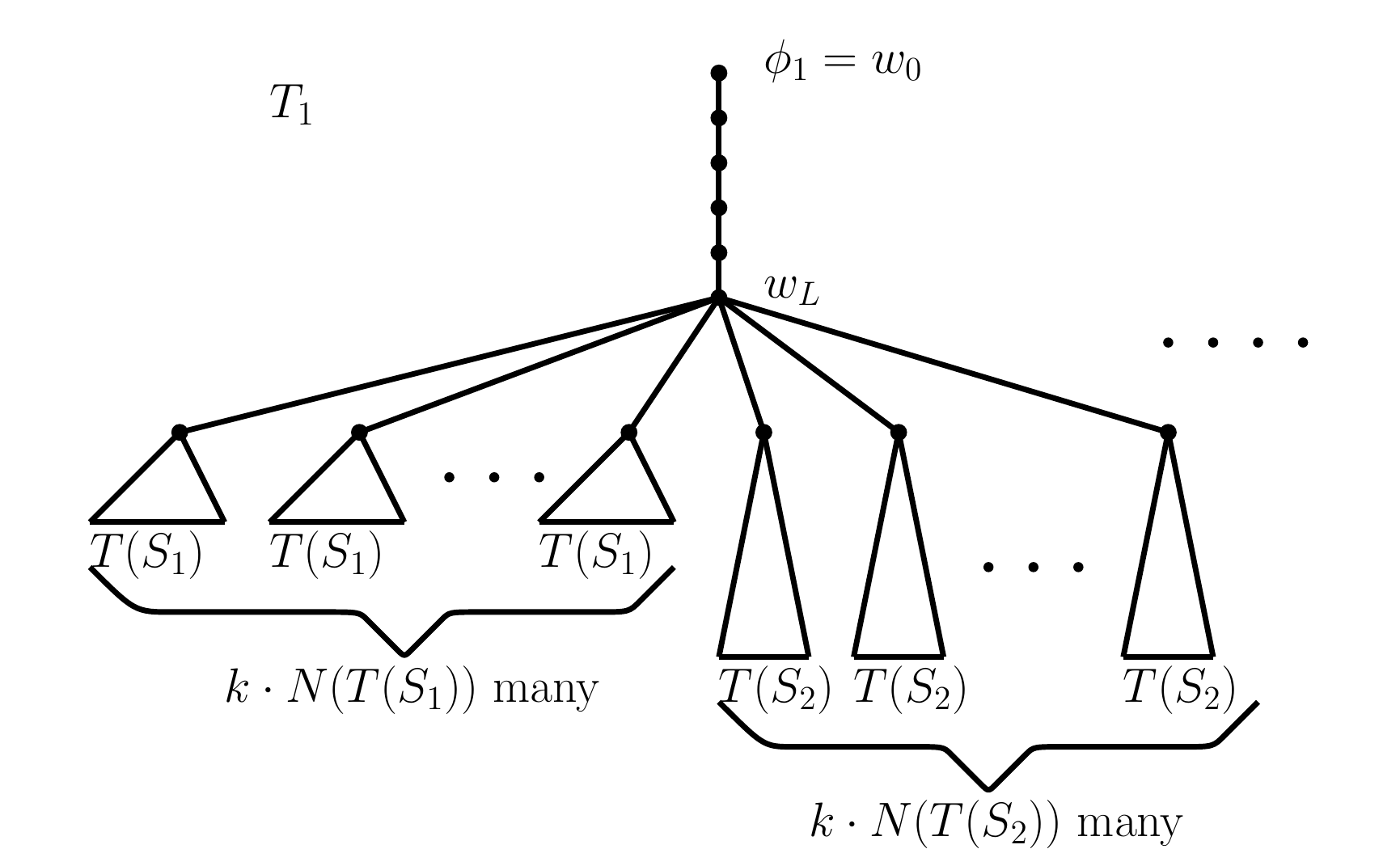}
  \caption{The finite tree $T_{1}$}
  \label{finite_image}
\end{figure}

\begin{figure}[h!]
  \centering
    \includegraphics[width=0.7\textwidth]{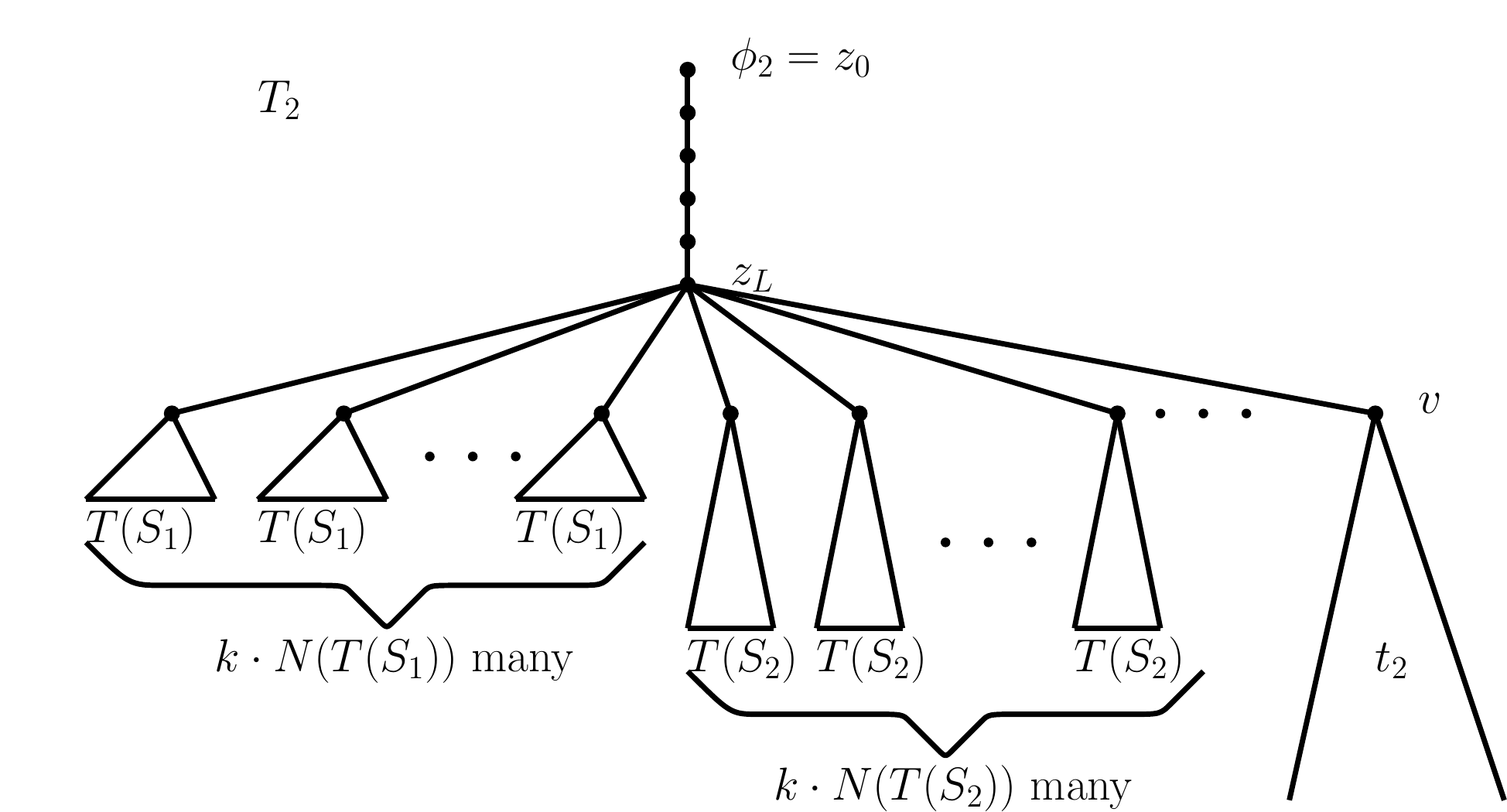}
  \caption{The infinite tree $T_{2}$ corresponding to $t_{2}$}
  \label{infinite_image}
\end{figure}

\par In the following section we give the proof of Theorem~\ref{main 3}. It consists of a few parts. In particular, we shall include a lemma as part of the proof. For this reason, we include this proof as a separate subsection.

\subsection{Proof of Theorem~\ref{main 3}:} Firstly, we fix any $T_{2} \in \mathcal{T}_{2}$. Suppose Spoiler assigns the $(\Sigma, \col_{0})$-rooted colouring $\sigma_{1}$ to $T_{1}$. By the pigeon hole principle, for each $S \in Q$, we must have at least one colouring $\sigma_{S}: V(T(S)) \rightarrow \Sigma \setminus \{\col_{0}\}$ that appears at least $k$ times among the copies of $T(S)$ in $T_{1}$, i.e.\ there exist $k$ positive integers $1 \leq i_{1}^{S} < i_{2}^{S} < \ldots < i_{k}^{S} \leq k \cdot N(T(S))$, such that 
\begin{equation}\label{many}
\Big(T_{1}\left(u_{S, i^{S}_{j}}\right), \sigma_{1}\Big) \cong \big(T(S), \sigma_{S}\big), \text{ for all } 1 \leq j \leq k,
\end{equation}
where the symbol $\cong$ means that there is a graph isomorphism between these two that preserves the roots and colours of nodes.

\par Corresponding to this $\sigma_{S}$, define a colouring $\widetilde{\sigma}_{S}: V(T(S)) \rightarrow \Sigma$ such that:
\begin{equation} \label{pseudo}
\widetilde{\sigma}_{S}(v) = 
  \begin{cases} 
   \col_{0} & \text{if } v \text{ is the root } \phi_{S} \text{ of } T(S), \\
   \sigma_{S}(v) & \text{otherwise }.
  \end{cases}
\end{equation}
Note that $\widetilde{\sigma}_{S}$ is a $(\Sigma, \col_{0})$-rooted colouring of $T(S)$.

\begin{lemma}\label{union ubi}
Fix the colour set $\Sigma$, the cut-off $k$ and the depth $m$. For every $S \subseteq Q$, we consider the coloured tree $\left(T(S), \widetilde{\sigma}_{S}\right)$. For every $v \in V(T(S))$, we consider its type with respect to this coloured tree. Define
\begin{equation}\label{closure of S}
\widehat{S} = \left\{\Type_{\left(T(S), \widetilde{\sigma}_{S}\right)}(v): v \in V(T(S))\right\},
\end{equation}
and the set $X = \bigcup_{S \in Q} \widehat{S}.$ Then $X$ is adequate.
\end{lemma}

\begin{proof}
Suppose not. Then $X \in Q$. Consequently, there exists no $(\Sigma, \col_{0})$-rooted colouring $\sigma$ of $T(X)$ such that, 
$$\left\{\Type_{(T(X), \sigma)}(v): v \in V(T(X))\right\} \subseteq X.$$
However, when we consider the the $(\Sigma, \col_{0})$-rooted colouring $\widetilde{\sigma}_{X}$ of $T(X)$, we indeed get
$$\left\{\Type_{(T(X), \widetilde{\sigma}_{X})}(v): v \in V(T(X))\right\} = \widehat{X} \subseteq X,$$ by \eqref{closure of S} and the definition of $X$. This brings us to a contradiction. Hence indeed $X$ is adequate.
\end{proof}

\par \textbf{Duplicator's response:} We construct a $(\Sigma, \col_{0})$-rooted colouring $\sigma_{2}$ on $T_{2}$, which will be Duplicator's winning response. For all $1 \leq i \leq L$, we set $\sigma_{2}(z_{i}) = \sigma_{1}(w_{i})$, and for all $S \in Q$, and all $1 \leq i \leq k \cdot N(T(S))$, we set
\begin{equation}\label{exact duplicate}
\left(T_{2}\left(v_{S, i}\right), \sigma_{2}\right) \cong \left(T_{1}\left(u_{S, i}\right), \sigma_{1}\right),
\end{equation}
where once again $\cong$ indicates an isomorphism between the coloured trees as previously described. We define $\sigma_{2}$ on $T_{2}(v)$, the infinite branch, in the following way:
\begin{enumerate}[label={(Step \arabic*)},leftmargin=*]
\item \label{step 1} Since $X$ is adequate by Lemma~\ref{union ubi}, there exists a $(\Sigma, \col_{0})$-rooted colouring $\widetilde{\sigma}_{2}$ of $T_{2}(v)$ such that 
\begin{equation}\label{good eq}
\left\{\Type_{(T_{2}(v), \widetilde{\sigma}_{2})}(u): u \in V\left(T_{2}(v)\right)\right\} \subseteq X.
\end{equation}
Note that this assigns colour $\col_{0}$ to $v$.
\item \label{step 2} Consider the type $\gamma$ of the node $v$ in $\left(T_{2}(v), \widetilde{\sigma}_{2}\right)$. Note that the type $\gamma$ has the colour $\col_{0}$ at its root, since $\widetilde{\sigma}_{2}$ is a $(\Sigma, \col_{0})$-rooted colouring. This type must belong to $X$ because of \eqref{good eq}. By definition of $X$ from Lemma~\ref{union ubi}, there exists some $S_{0} \in Q$ such that $\gamma \in \overline{S_{0}}$. By \eqref{closure of S}, this means that there must exist some node in $V(T(S_{0}))$ whose type is $\gamma$ with respect to the colouring $\widetilde{\sigma}_{S_{0}}$. Since $\widetilde{\sigma}_{S_{0}}$ is a $(\Sigma, \col_{0})$-rooted colouring of $T(S_{0})$, the only node that has the colour $\col_{0}$ is the root $\phi_{S_{0}}$ of $T(S_{0})$, hence $\phi_{S_{0}}$ must be the node with type $\gamma$. Thus:
\begin{equation}\label{gamma}
\gamma = \Type_{\left(T(S_{0}), \widetilde{\sigma}_{S_{0}}\right)}\left(\phi_{S_{0}}\right).
\end{equation}
Duplicator now defines the following assignment on $T_{2}(v)$:
\begin{equation} \label{on infinite subtree}
\sigma_{2}(u) = 
  \begin{cases} 
   \sigma_{S_{0}}\left(\phi_{S_{0}}\right) & \text{if } u = v, \\
   \widetilde{\sigma}_{2}(u) & \text{if } u \in T_{2}(v) \setminus \{v\}.
  \end{cases}
\end{equation} 
\end{enumerate}

\subsection{Justifying that this is a winning response:} We have to verify \eqref{win type} holds. We shall first show that \eqref{win type} holds for the coloured subtrees $\Big(T_{1}\left(w_{L}\right), \sigma_{1}\Big)$ and $\Big(T_{2}\left(z_{L}\right), \sigma_{2}\Big)$, and also show that the types of $w_{L}$ and $z_{L}$ are the same. This will then enable us to conclude that the types of $w_{i}$ and $z_{i}$ are the same for each $1 \leq i \leq L-1$, as well as those of the roots $\phi_{1}$ and $\phi_{2}$. This will conclude our verification.

\par For any $\eta \in \Gamma$, let $m^{(1)}_{\eta}$ denote the number of nodes $u$ in $T_{1}\left(w_{L}\right) \setminus \{w_{L}\}$ such that $\Type(u) = \eta$ with respect to $(T_{1}, \sigma_{1})$; let $m^{(2)}_{\eta}$ be the corresponding number in $T_{2}\left(z_{L}\right) \setminus \{z_{L}\}$. Because of \eqref{exact duplicate}, we only need to worry about the types $\eta$ that appear in $\left(T_{2}(v), \sigma_{2}\right)$, and the types of the nodes $w_{L}$ and $z_{L}$. We divide our analysis into the following three cases:

\begin{enumerate}
\item \label{need later} The type $\eta$ is the type of $v$ itself. By \eqref{gamma} and \eqref{on infinite subtree}, it is immediate that $\eta$ is then the type of $\phi_{S_{0}}$ with respect to $\sigma_{S_{0}}$ (since $\eta$ is simply $\gamma$ with the colour at the root switched from $\col_{0}$ to $\sigma_{S_{0}}\left(\phi_{S_{0}}\right)$). From \eqref{many}, it follows that $\eta$ is the type of $u_{S_{0}, i^{S_{0}}_{j}}$ for every $1 \leq j \leq k$. Therefore, $m^{(1)}_{\eta} \geq k$. Moreover, by \eqref{exact duplicate}, we know that $\eta$ is also the type of $v_{S_{0}, i^{S_{0}}_{j}}$ for every $1 \leq j \leq k$. Hence, we have $m^{(2)}_{\eta} \geq k+1$. Therefore, \eqref{win type} holds for $\eta$.   
\item The type $\eta$ appears in $\big(T_{2}(v), \sigma_{2}\big)$, but is not the type of $v$ itself. Then $\eta \in X$, which implies that $\eta \in \widehat{S}$ for some $S \in Q$. By \eqref{closure of S}, this means that there exists some $u \in V(T(S))$ such that the $\Type(u) = \eta$ with respect to $\big(T(S), \widetilde{\sigma}_{S}\big)$. Since $\eta$ does not have $\col_{0}$ at the root, hence $u$ is not the root $\phi_{S}$ of $T(S)$. Hence, we actually can conclude that 
$$\Type_{\big(T(S), \sigma_{S}\big)}(u) = \eta.$$
By \eqref{many}, this tells us that for $1 \leq j \leq k$, the subtree $\Big(T_{1}\left(u_{S, i^{S}_{j}}\right), \sigma_{1}\Big)$, being isomorphic to $\big(T(S), \sigma_{S}\big)$, contains at least one node of type $\eta$. Consequently, $m^{(1)}_{\eta} \geq k$. By \eqref{exact duplicate}, each copy $\left(T_{2}\left(v_{S, i^{S}_{j}}\right), \sigma_{2}\right)$ will also contain at least one occurrence of the type $\eta$, and $\big(T_{2}(v), \sigma_{2}\big)$ contains at least one more. Hence $m^{(2)}_{\eta} \geq k+1$. Hence, once again, \eqref{win type} holds.
\item \label{three} Finally, we show that the types of $w_{L}$ and $z_{L}$ are the same. Note that, by \eqref{exact duplicate}, the types of $u_{S, i}$ and $v_{S, i}$ are the same for every $S \in Q$ and every $1 \leq i \leq k \cdot N(T(S))$. For the type $\eta$ of $v$, in \eqref{need later}, we have already shown that $w_{L}$ has at least $k$ children $u_{S_{0}, i^{S_{0}}_{j}}$, $1 \leq j \leq k$, with type $\eta$, and $z_{L}$ has at least $k+1$ children: $v_{S_{0}, i^{S_{0}}_{j}}$, $1 \leq j \leq k$, and $v$, with type $\eta$. Consequently, the number of children of each $(\Sigma, m, k)$-type, truncated at $k$, of both $w_{L}$ and $z_{L}$ are the same. By the recursive definition of types in Definition~\ref{type def}, this shows that $w_{L}$ and $z_{L}$ have the same $(\Sigma, m+1, k)$-type, hence also the same $(\Sigma, m, k)$-type. 
\end{enumerate}

This completes the verification that \eqref{win type} holds for the subtrees $\left(T_{1}\left(w_{L}\right), \sigma_{1}\right)$ and $\left(T_{2}\left(z_{L}\right), \sigma_{2}\right)$. Having shown in \eqref{three} that $w_{L}$ and $z_{L}$ have the same $(\Sigma, m, k)$-types, the recursive definition of types in Definition~\ref{type def} and the fact that we have defined $\sigma_{2}(z_{i}) = \sigma_{1}(w_{i})$ for every $0 \leq i \leq L-1$, allow us to conclude that the nodes $w_{i}$ and $z_{i}$ have the same $\left(\Sigma, m+L-i, k\right)$-types. Hence they also have the same $(\Sigma, m, k)$-types. 

This brings us to the end of the justification that indeed, Duplicator wins $\Types\left[T_{1}, T_{2}, \Sigma, m, k\right]$. It is now immediate that combining Theorem~\ref{main 2}, Theorem~\ref{main 3} and Remark~\ref{positive probab}, we get the desired Theorem~\ref{main 1}.

\bibliography{finiteEMSO}

\end{document}